\documentclass[11pt]{amsart}
\usepackage{amssymb}
\usepackage{amsmath}
\usepackage{nicefrac}
\usepackage{enumerate}
\usepackage{enumitem}
\usepackage{hyperref}
\usepackage{float}
\usepackage{xcolor}
\usepackage{MnSymbol}
\usepackage{xparse}
\usepackage{tabularx}
\usepackage{tikz}
\usepackage{mathtools}
\usepackage{dsfont}

\theoremstyle{plain}

\newtheorem{thm}{Theorem}[section]
\newtheorem{lemma}[thm]{Lemma}
\newtheorem{cor}[thm]{Corollary}
\newtheorem{prop}[thm]{Proposition}
\newtheorem{prob}[thm]{Problem}

\theoremstyle{definition}
\newtheorem{defn}[thm]{Definition}
\newtheorem{ex}[thm]{Example}
\newtheorem{rmk}[thm]{Remark}
\newtheorem{warn}[thm]{Warning}

\newtheorem*{assumption*}{Assumption}

\theoremstyle{remark}
\newtheorem*{claim*}{Claim}

\def\bT{\mathbf{T}}
\def\sX{\mathsf{X}}
\def\sT{\mathsf{T}}
\def\AT{W}
\def\Aut{\operatorname{Aut}}
\def\Fix{\operatorname{Fix}}
\def\Ker{\operatorname{Ker}}
\def\St{\operatorname{St}}

\def\Sym{\mathcal{S}}
\def\Alt{\mathcal{A}}
\def\shift{\mathsf{s}}
\def\dist{\operatorname{dist}}
\newcommand{\one}{\mathds{1}} 
\newcommand{\rs}{{\mathbb{P}^1(\mathbb {C})}}
\newcommand{\post}{P}
\newcommand{\crit}{C}
\newcommand{\IMG}{\Gamma^{discr}}
\newcommand{\MM}{\mathcal{M}}
\newcommand{\PP}{\mathcal{P}}

\newcommand{\Z}{\mathbb{Z}}
\newcommand{\C}{\mathbb{C}}

\newcommand{\sgn}{\text{sgn}}
\newcommand{\supp}{\operatorname{Supp}}

\let\emptyset\varnothing

\makeatletter
\NewDocumentCommand{\caseitem}{m}{%
  \phantomsection%
  \def\@currentlabel{#1}%
  \label{case:#1}%
  \textrm{\bf Case~#1:}%
}
\makeatother

\input xy
\xyoption {all}

\date{\today}
 
\title[Profinite geometric iterated monodromy groups in degree $3$]{Profinite geometric iterated monodromy groups of postcritically finite polynomials in degree $3$}
\date{\today}

\author[]{Mikhail Hlushchanka}
\address{Mikhail Hlushchanka, Korteweg-de Vries Institute for Mathematics, University of Amsterdam, P.O. Box 94248, 1090 GE Amsterdam, The Netherlands \\mikhail.hlushchanka@gmail.com, ORCID: 0000-0002-2450-8023}

\author[]{Olga Lukina}
\address{Olga Lukina, Mathematical Institute, Leiden University, P.O. Box 9512,
2300 RA Leiden,
The Netherlands \\o.lukina@math.leidenuniv.nl, ORCID: 0000-0001-8845-3618}

\author[]{Dean Wardell}
\address{Dean Wardell, Mathematical Institute, Leiden University, P.O. Box 9512,
2300 RA Leiden,
The Netherlands \\
d.a.wardell@math.leidenuniv.nl, ORCID: 0009-0007-2176-2409}



\thanks{M.H. was supported by the Marie Skłodowska-Curie Postdoctoral Fellowship
  under Grant No.\ 101068362.}

\begin{document}

\maketitle

\begin{abstract}
In this article, we study the properties of profinite geometric iterated monodromy groups associated to polynomials. Such groups can be seen as generic representations of absolute Galois groups of number fields into the automorphism group of a regular rooted tree. Our main result is that, for a degree $3$ postcritically finite polynomial over a number field, where each finite postcritical point has at least one preimage outside the critical orbits, the associated profinite geometric iterated monodromy group is finitely invariably generated. Moreover, this group is determined by the isomorphism class of the ramification portrait of the polynomial, up to conjugation by an automorphism of the ternary rooted tree. 
We also study the group-theoretical properties of such groups, namely their branch and torsion properties. In particular, we show that such groups are regular branch over the closure of their commutator subgroup, and that they contain torsion elements of any order realizable in the ternary tree.
\end{abstract}

\section{Introduction}\label{sec-introduction}

 In this article, we study profinite geometric iterated monodromy groups associated to cubic polynomials over number fields that have two distinct finite critical points with finite orbits, from a group-theoretical point of view. We show that, for polynomials satisfying Assumption~\ref{main-assumption} below, each such profinite group is finitely invariably generated and, moreover, it is determined up to an isomorphism by the combinatorial characteristics of the critical orbits. We provide (canonical) models for our profinite geometric iterated monodromy groups and use this description to study their properties: we show that they are regular branch and have torsion elements of all orders realizable by automorphisms of a ternary rooted tree. Our results apply to all postcritically finite polynomials of degree three such that every finite postcritical point has a preimage outside the critical orbits; see Assumption~\ref{main-assumption}.

Similar results for profinite geometric iterated monodromy groups associated to \emph{unicritical} polynomials of degree $d \geq 2$ were obtained by Pink \cite{Pink2013} in the quadratic case and by Adams and Hyde \cite{AdamsHyde} in the general case. In both \cite{AdamsHyde,Pink2013}, an invariable generating set for such groups was provided by their standard generators associated with \emph{finite} postcritical points of the given polynomial. This is no longer true in our setting, where, to obtain an invariable generating set of standard generators, one must include the generator corresponding to the \emph{infinite} critical point. This is a new phenomenon, not present in \cite{AdamsHyde,Pink2013}.

\subsection{Profinite iterated monodromy groups}\label{subsec: profinite-imgs}
We start by recalling the construction of such groups for polynomials over number fields; see \cite[Sec.~2]{CL2022} for more details and \cite{Jones-survey} for the discussion of the general case of rational maps over a global field.

Let $K$ be a number field, and let $f$ be a rational map of degree $d \geq 2$ with coefficients in $K$. Here we consider $f$ as the map $f\colon \mathbb{P}^1(K)\to \mathbb{P}^1(K)$ of the projective line over $K$, so that, counting multiplicity, $f$ has $2d-2$ critical points. If $f$ is a polynomial, one of its critical points is the point at infinity, denoted $\infty$, and we refer to $\infty$ as the \emph{infinite} critical point, and to the other critical points of $f$ as \emph{finite} critical points. Then, counting multiplicity, there are $d-1$ finite critical points, and $\infty$ has multiplicity $d-1$. For $n \geq 0$, we write $f^n$ for the $n$-th iterate of the rational map $f$, where $f^0 = {\rm id}$. We denote by $C(f)$ the set of distinct critical points of $f$, and by $\post(f)=\bigcup_{n\geq 1} f^n\big(\crit(f)\big)$ the \emph{postcritical set} of $f$. The rational map $f$ is called \emph{postcritically finite}, or \emph{PCF}, if $\post(f)$ is a finite set.

We assume now that $f$ is a polynomial over the number field $K$. Let $t$ be a transcendental element over $K$, and let $K(t)$ be the field of rational functions with coefficients in $K$. For each $n\geq 0$, the roots of the equations $f^n(z) =t$ are adjoined to $K(t)$ to create an increasing sequence of finite degree field extensions $K_n$ of $K(t)$, and an infinite degree field extension $\mathcal K := \bigcup_{n \geq 1} K_n$. 
 Polynomials $f^n(z) - t$ are separable and irreducible over $K(t)$ \cite{AHM2005}, and thus, for each $n \geq 0$, the Galois group $H_n:=\operatorname{Gal}\big(K_n/K(t)\big)$ acts transitively on the $d^n$ distinct roots of $f^n(z) - t$.

 The \emph{profinite arithmetic iterated monodromy group} $G^{arith}(f)$ associated to $f$ is obtained as a profinite subgroup of the automorphism group of a $d$-ary rooted tree $T$, which we introduce below.

\begin{defn}\label{defn-d-ary}
      A \emph{$d$-ary rooted tree} $T$, for $d \geq 2$, is an infinite graph without cycles, which has the set of vertices $V = \bigsqcup_{n \geq 0} V_n$, such that for all $n \geq 0$ we have the following: $|V_n| = d^n$, every $v \in V_n$ is connected by edges to precisely $d$ vertices in $V_{n+1}$, and every $v' \in V_{n+1}$ is connected by an edge to exactly one vertex in $V_n$. The finite sets $V_n$ are called the \emph{level sets}, or \emph{levels}, of the tree $T$, and the single vertex in $V_0$ is called the \emph{root} of $T$. We denote by $\Aut(T)$ the group of automorphisms of $T$, i.e., the maps that preserve the structure of~$T$.
\end{defn}

 Let $T$ be a $d$-ary rooted tree, where $d=\deg(f)$, with the vertex set $V = \bigsqcup_{n \geq 0} V_n$ as above.
 We identify the root in $V_0$ with $t$, and for each $n \geq 1$, we identify $V_n$ with the set of the solutions of $f^n(z) = t$ so that $v' \in V_{n}$ and $v \in V_{n-1}$ are connected by an edge if and only if $f(v') = v$. Since $K_{n-1}\subset K_{n}$, we have an induced epimorphism $H_{n}\to H_{n-1}$, and the induced action of $H_n$ on $\bigsqcup_{i=0}^nV_i$ preserves the relation of being connected by an edge. Thus we obtain an injective homomorphism 
  $$\rho\colon {\rm Gal}\big(\mathcal K/K(t)\big) \to \Aut(T),$$
from the Galois group of $\mathcal K$ into the automorphism group of $T$ and define the profinite arithmetic iterated monodromy group as its image
  $$G^{arith}(f) := \rho\big({\rm Gal}\big(\mathcal K/K(t)\big)\big) \cong \varprojlim\{H_n \to H_{n-1}, \, n\geq 1\}.$$
The group $G^{arith}(f)$ is a separable profinite topological group, and it can be thought of as a generic \emph{arboreal representation} of the absolute Galois group of the number field $K$ into ${\rm Aut}(T)$, via a specialization to a value of $t$, see \cite{Jones-survey}. Thus the study of absolute Galois groups of number fields motivates the interest in the study of profinite arithmetic iterated monodromy groups.

Arboreal representations for quadratic polynomials, over various fields, have been studied extensively. Profinite arithmetic iterated monodromy groups in the quadratic case were described in detail by Pink \cite{Pink2013}, who used the approach via profinite geometric iterated monodromy groups, applicable also to any rational map of degree $d \geq 2$.

\begin{defn}\label{defn-geometric-pimg}
The \emph{profinite geometric iterated monodromy group} $G^{\text{geom}}(f)$ is defined as the image $\rho \big({\rm Gal}\big(\mathcal K/L(t)\big)\big) \subset \Aut(T)$, where $L$ is the subfield of $\mathcal{K}$ consisting of all elements algebraic over $K$, i.e., $L = \overline{K} \cap \mathcal{K}$ with $\overline{K}$ denoting a separable closure of $K$.
\end{defn}

Since the profinite group $G^{geom}(f)$ does not change if the field $L$ is extended, it can be calculated over $\mathbb C(t)$ \cite{Jones-survey}. Moreover, when $f$ is PCF, there is a natural isomorphism between $G^{geom}(f)$ and the profinite completion of the \emph{discrete iterated monodromy group} $\IMG(f)$ associated to $f$; see Section~\ref{subsec: img-discr} and Proposition~\ref{prop: discrete-vs-geometric-img} in particular.

The group $G^{arith}(f)$ contains $G^{geom}(f)$ as a normal subgroup. Using the theory of \'etale fundamental groups, one can choose a \emph{standard set} of topological generators $S=\{g_p: p\in\post(f)\}\subset \Aut(T)$ of $G^{geom}(f)$, with each generator corresponding to a postcritical point. If $f$ is PCF and $K$ is a number field, to define the standard generators $g_p$, one can use instead the fundamental group of $\rs\setminus \post(f)$; see Section~\ref{subsec-generators}. In either approach, each generator $g_p\in S$ is determined up to conjugation in $G^{geom}(f)$, and its action, described using wreath recursion, is specified (up to conjugation in $\Aut(T)$) by the ramification portrait of $f$, see Definition~\ref{defn-ramification-portrait}.  The results in \cite{Pink2013} state that the profinite group $G^{geom}(f)$ for a quadratic polynomial $f$ is determined, up to isomorphism, by only combinatorial characteristics of the orbit of the unique finite critical point of $f$: the length of the orbit, whether the orbit is strictly periodic, and the length of the periodic cycle. This makes it possible to choose a set of \emph{model generators} for $G^{geom}(f)$, and study in detail the properties of $G^{geom}(f)$ and $G^{arith}(f)$. Using this approach, Pink also described $G^{geom}(f)$ and $G^{arith}(f)$ in the case when $f$ is a quadratic rational map with an infinite postcritical set  \cite{Pink2013a}.

Cortez and Lukina \cite{CL2022} used the results of \cite{Pink2013} in the study of settled elements, giving a partial answer to the conjecture by Boston and Jones \cite{BJ2007} about images of Frobenius elements under arboreal representations. Ejder \cite{Ejder2024} used the approach of \cite{Pink2013} to discover the first example of a profinite arithmetic iterated monodromy group that does not contain an odometer. Arboreal representations for the polynomial $f(z) = -2z^3+3z^2$ over a large class of number fields were described by Benedetto et al. \cite{BFHJY2017}. The polynomial in \cite{BFHJY2017} is a representative of the class of normalized Belyi maps, whose profinite iterated monodromy groups were later studied by Bouw, Ejder, and Karemaker \cite{BEK2018} and by Ejder \cite{Ejder2022}.  
Arboreal representations for cubic polynomials with infinite orbits of the finite critical points colliding in the $\ell$-th iterate (with $\ell\geq 2$) were studied by Benedetto et al. \cite{Ben-al2025}. Necessary and sufficient conditions for an arboreal representation to have finite index in $\Aut(T)$, for cubic polynomials over various fields, were obtained by Bridy and Tucker \cite{BT2019}. Recent work by Adams and Hyde \cite{AdamsHyde} implements Pink's program for unicritical PCF polynomials (that is, for degree $d\geq 2$ PCF polynomials with a single finite critical point of multiplicity $d-1$).

In this paper, we make a step towards generalizing the results and techniques of \cite{Pink2013} beyond the unicritical case by providing an explicit description of the profinite geometric iterated monodromy group $G^{geom}(f)$ for PCF cubic polynomials $f$ with two distinct finite critical points, such that every finite postcritical point of $f$ has a preimage outside the critical orbits. 
Generalizing the program in \cite{Pink2013} to higher degrees, already at the first steps, one encounters technical difficulties which do not arise in the quadratic case. First, the possible combinatorial types of the (potentially colliding) critical orbits become much more diverse, compared to the only two alternatives (strictly preperiodic or periodic) in the unicritical setting.  Moreover, the results in \cite{Pink2013} depend on the property of such groups called \emph{semirigidity}: for a given $f$, the action of each element of the standard generating set $S$ of $G^{geom}(f)$ is specified up to conjugation in $\Aut(T)$, and semirigidity is the property that, after replacing each element in $S$ by a conjugate in $\Aut(T)$ (with a conjugating element depending on $g_p \in S$), one obtains a subgroup of $\Aut(T)$ conjugate to $G^{geom}(f)$. The property of semirigidity is closely related to the problem of \emph{invariable generation} in countable and profinite groups: given a group $G$ with a generating set $S$, the set $S'$ of pairwise conjugates of the elements of $S$ by different elements of $G$ may not be a generating set for $G$. This led to the development of the notion of an \emph{invariable generating set} of a group, see Definition~\ref{defn-inv-gen} below. Here, if $G$ is a profinite group, then $S$ is understood to be a set of topological generators of $G$, that is, $S$ generates a subgroup of $G$ that is dense in $G$ in the profinite topology.

\begin{defn}\label{defn-inv-gen}
   Let $G$ be a group. A set $S \subset G$ 
    \emph{invariably generates} $G$ if, for any choice of $g_s \in G$, the set $\{g_s s g_s^{-1}: s \in S\}$ generates $G$.   
    A group $G$ is (\emph{countably}, resp. \emph{finitely}) \emph{invariably generated} if there is a (countable, resp. finite) set $S \subset G$ that invariably generates~$G$.
\end{defn}

The study of invariable generation is currently an active area of research in group theory, see \cite{KLS2015,Lucchini2017} and references therein. An example of a countable group which is not invariably generated is the free group of rank $k \geq 2$  \cite{Wiegold1976},  and there exist many others, see, for instance, \cite{KLS2015,Wiegold1977}. Profinite groups are invariably generated but need not be finitely invariably generated \cite[Sec.~4]{KLS2015}. For instance, Lucchini \cite{Lucchini2017} proves that, for a given positive integer $k$, the infinite iterated wreath product of finite cyclic groups of prime orders, with each order repeating at most $k$ times, is not finitely invariably generated, even though  it is (topologically) finitely generated. Kantor, Lubotzky and Shalev \cite[Thm~1.5]{KLS2015} show that the profinite completion of any Fuchsian group is not finitely invariably generated. Moreover, there exist groups whose profinite completions are finitely invariably generated, but the groups themselves are not \cite{KLS2015}.

In the setting of profinite geometric iterated monodromy groups, if $f$ is a quadratic polynomial and $T$ is a binary rooted tree, then $G^{geom}(f) \subset \Aut(T)$ is approximated by an infinite sequence of finite $2$-groups, which are nilpotent. Invariable generation in nilpotent groups is well-studied, and the nilpotent properties of $2$-groups are the basis of the argument in \cite[Lem.~1.3.2]{Pink2013}, which shows that $G^{geom}(f)$ is finitely invariably generated. 
This argument does not generalize to degree $3$, apart from the case when $f$ is unicritical, as the finite groups approximating $G^{geom}(f)$ are no longer nilpotent. For unicritical PCF polynomials in arbitrary degree $d \geq 2$, the invariable generation result follows from \cite[Lem.~4.19 and Cor.~4.23]{AdamsHyde}. In both settings, one can choose the standard generating set $S$ of  $G^{geom}(f)$ minus the generator corresponding to $\infty$ as an invariable generating set.

\medskip
\subsection{Main results} In this section, $T$ always denotes a ternary rooted tree. 
We distinguish classes of polynomials using their \emph{ramification portrait}, a notion that can also be introduced for rational maps. 

\begin{defn}\label{defn-ramification-portrait}
    Let $f:\mathbb P^1(\mathbb{C})\to \mathbb P^1(\mathbb{C})$ be a rational map. The \emph{ramification index} of $f$ at $p\in \mathbb{P}^1(\mathbb C)$ is the natural number $\deg(f,p)$ that records the multiplicity of the zero $p$ of the map $z\mapsto f(z)-f(p)$, after a suitable coordinate change. The \emph{ramification portrait} of $f$ is the directed graph $\mathcal P_f$ with vertex set $V(\mathcal P_f):=C(f)\bigcup P(f)$, such that for each $p\in V(\mathcal P_f)$ there are $\deg(f,p)$ edges from $p$ to $f(p)$. 
\end{defn}

\begin{assumption*} 
    In this paper, we consider PCF cubic polynomials $f$ such that:
    
    \begin{enumerate}[label=({\rm Y})] 
    \item\label{main-assumption} The ramification portrait $\PP_f$ of $f$ has exactly one vertex with three incoming edges. 
\end{enumerate}
\end{assumption*}

Assumption \ref{main-assumption} is equivalent to saying that for each finite postcritical point $p\in P(f)\setminus \{\infty\}$ we have 
$$f^{-1}(p)\not\subset C(f)\cup P(f).$$ 
That is, each such point $p$ has a preimage under $f$ that is neither critical nor postcritical. In particular, if $f$ is a polynomial and Assumption~{\normalfont \ref{main-assumption}} is satisfied, then the point  with three incoming edges is the infinite critical point $\infty$. Consequently,  unicritical polynomials of degree $3$ do not satisfy Assumption~{\normalfont \ref{main-assumption}}.

\begin{thm}\label{main-theorem}
    Let $f$ be a PCF polynomial map of degree $3$ over a number field $K$ with postcritical set $\post(f)$, and let 
     $$S  = \{g_p : p \in \post(f)\} \subset \Aut(T)$$ 
     be a standard generating set for the profinite geometric iterated monodromy group $G^{geom}(f)$, as in Definition~\ref{def: standard-generators}. 
    
    Suppose $f$ satisfies Assumption \emph{\ref{main-assumption}}. Then the following is true:
    
    \begin{enumerate}[label=(\roman*),font=\normalfont]
        \item \label{item-MT-1}The profinite geometric iterated monodromy group $G^{geom}(f)$ is finitely invariably generated, with invariable generating set  
      $S$. In particular, $S$ includes $g_\infty$, the generator corresponding to $\infty$, which acts transitively on every level of the tree $T$.
        \item \label{item-MT-2} Let $f'$ be another polynomial of degree 3 over a number field, and suppose the ramification portraits of $f$ and $f'$ are isomorphic as graphs. Then there exists some $w\in \Aut(T)$ that conjugates the profinite geometric iterated monodromy groups associated to $f$ and $f'$, that is,
        $$G^{geom}(f)=wG^{geom}(f')w^{-1}.$$
   \end{enumerate} 
\end{thm}

\begin{rmk} 
    Our approach to defining the standard generators for $G^{geom}(f)$ is via discrete iterated monodromy groups, and our results are stated for groups associated to polynomials over number fields. Adams and Hyde \cite{AdamsHyde} work in a more general setting, by defining and using abstract algebraic \emph{paths}, see \cite[Def.~3.4]{AdamsHyde}, instead of topological \emph{connecting paths} from Section~\ref{subsec: self-similarity}, as well as algebraic \emph{inertia generators} over  postcritical points $p\in\post(f)$ instead of topological \emph{$p$-petals} from Section~\ref{subsec-generators}. Lemmas~3.11 and 3.12 from \cite{AdamsHyde} provide analogues of the structural statements about the standard generators of $G^{geom}(f)$ from Proposition~\ref{prop: img-gen-structure}, in their more general setting. In the unicritical cubic case, this allows them to consider polynomials over fields of any characteristic coprime to $3$, rather than only over number fields. Therefore, we expect that our results in Theorems~\ref{main-theorem}, \ref{thm-reg-branch}, \ref{thm-filtration} and Corollary~\ref{polynomial-branch} extend to polynomials over fields of characteristic coprime to $2$ and $3$.
\end{rmk}

To prove Theorem~\ref{main-theorem},\ref{item-MT-2}, we assign to each PCF cubic polynomial $f$ satisfying Assumption~\ref{main-assumption} a canonical \emph{\ref{main-assumption}-restricted model group} $G^{model}(f)$; see Section~\ref{Section:Model gps}, and especially Proposition~\ref{prop: model-group-from-map}. By this we mean that we provide an explicit set of generators $S^{model}(f)\subset \Aut(T)$, described using wreath recursions, such that $G^{geom}(f)$ is conjugate to the profinite group $G^{model}(f)$ topologically generated by $S^{model}(f)$. The \ref{main-assumption}-restricted model group $G^{model}(f)$ is unique up to relabelling the finite postcritical points in the construction of $S^{model}(f)$, and it is determined entirely by the ramification portrait of $f$. 

We emphasize that the conjugacy in Theorem \ref{main-theorem},\ref{item-MT-2}, holds \emph{only} for the profinite groups $G^{geom}(f)$ and $G^{geom}(f')$ associated to the polynomials $f$ and $f'$ with isomorphic ramification portraits, but \emph{need not} hold for the countable dense subgroups of $G^{geom}(f)$ and $G^{geom}(f')$, such as the discrete iterated monodromy groups $\IMG(f)$ and $\IMG(f')$. This is a subtlety that arises when working with profinite groups. Indeed, the following problem goes back to Grothendieck \cite{Grothen}: let $h\colon \Gamma_1 \to \Gamma_2$ be a homomorphism of finitely presented residually finite countable groups, and suppose $G_i$ contains $\Gamma_i$ as a dense subgroup, for $i = 1,2$. Assume that the induced homomorphism $\widehat h\colon G_1 \to G_2$ is an isomorphism. Must $h$ be an isomorphism? This problem was answered in the negative for finitely generated groups in \cite{BL2000,PT1986,P2004}, and for finitely presented groups in \cite{BG2004}. For the automorphism group of the binary rooted tree, Nekrashevych \cite{Nekr2014} constructed an uncountable family of subgroups (with $3$ generators) that are pairwise non-isomorphic but have isomorphic profinite completions. In Theorem~\ref{main-theorem},\ref{item-MT-2}, although the two profinite geometric iterated monodromy groups $G^{geom}(f)$ and $G^{geom}(f')$ are conjugate, the conjugating isomorphism need not map the countable subgroup $\IMG(f)$ onto the subgroup $\IMG(f')$, and the properties of these subgroups may be different. Explicit cubic polynomials exhibiting this phenomenon can be constructed.

\medskip
A question related to the question of Grothendieck discussed above is: when does a property of a profinite group induces a similar property in a countable dense subgroup? For instance, it is known that, unless extra assumptions are satisfied (e.g., the countable group is nilpotent \cite{RibZal}), the non-triviality of the torsion subgroup of a profinite group need not imply that the torsion subgroup of its countable dense subgroup is non-trivial; see \cite{DHL2017,Lubotsky1993} for examples. In Theorem~\ref{thm-reg-branch} and Corollary~\ref{polynomial-branch}, we study the branch and torsion properties for, respectively, \ref{main-assumption}-restricted model groups and profinite geometric iterated monodromy groups associated to cubic PCF polynomials satisfying Assumption~\ref{main-assumption}.
We refer to Definition~\ref{weakly-branch} for the definition of a (weakly) regular branch group. 

\begin{thm}\label{thm-reg-branch}
    Let $G$ be a \emph{\ref{main-assumption}}-restricted model group, topologically generated by $a,b,c_1,\ldots,c_r\in \Aut(T)$, with $r\geq 0$, as defined in Definition~\ref{defn-group}, and let $\Gamma:=\langle a,b,c_1,\ldots,c_r\rangle$ be the corresponding countable dense subgroup. Then:
    \begin{enumerate}[label=(\roman*),font=\normalfont]
        \item \label{item1-branch} $\Gamma$ is regular weakly branch over its commutator subgroup $[\Gamma,\Gamma]$.
        \item \label{item2-branch} $G$ is regular branch over the subgroup $\mathfrak C := \overline{[G,G]}$. In particular, the index of $\mathfrak C$ in $G$ is at most $2^{r+2}$, where $r+2$ is the number of topological generators of $G$.
          \item \label{item3-torsion} For each $n \geq 0$, the countable group $\Gamma$ contains an element of order $3^n$.
        \item \label{item4-torsion} For each $m,n \geq 0$, the profinite group $G$ contains an element of order $2^m3^n$. Thus, all torsion orders realizable in $\Aut(T)$ are represented in $G$.
        
    \end{enumerate}
\end{thm}

\begin{cor}\label{polynomial-branch}
    Let $f$ be a PCF polynomial map of degree $3$ over a number field $K$. Suppose $f$ satisfies Assumption \emph{\ref{main-assumption}}. Then: 
    \begin{enumerate}[label=(\roman*),font=\normalfont]
    \item \label{item-poly-branch} The profinite geometric iterated monodromy group $G^{geom}(f)$ is regular branch over the closure of its commutator subgroup.
     \item \label{item-poly-torsion} For any integers $m,n\geq 0$, the group $G^{geom}(f)$ contains an element of order $2^m3^n$. Thus, all torsion orders realizable in $\Aut(T)$ are represented in $G^{geom}(f)$.
       \end{enumerate}
\end{cor}

Finally, we study profinite geometric iterated monodromy groups associated to cubic PCF polynomials that satisfy Assumption~\ref{main-assumption} and whose finite critical points have disjoint orbits. In this case, we prove that, under certain divisibility conditions on the lengths of the periodic and pre-periodic parts of the postcritical orbits, such groups form filtrations in $\Aut(T)$.

\begin{thm}\label{thm-filtration}
Let $f,f'$ be two PCF polynomials over number fields,  each satisfying Assumption~\emph{\ref{main-assumption}} and having disjoint critical orbits. For $f$, let $s_i$ and $m_i$ denote, respectively, the lengths of the pre-periodic and periodic parts of the postcritical orbits of its two finite critical points of $f$, indexed by $i=1,2$; see Figure~\ref{fig:disjoint orbits} for reference. Define $s_i',m_i'$ for $f'$ similarly.\\
Suppose that for each $i \in \{1,2\}$ there exists some $j \in \{1,2\}$ such that one of the following holds:
\begin{itemize}
        \item if $s_i'=0$, then $m_i'$ divides both $s_j$ and $m_j$;
        \item if $s_i'\neq 0$, then $(s_i',m_i')=(s_j,m_j)$.
    \end{itemize}
Then there exists some $w\in W$ such that $wG^{geom}(f')w^{-1}\subseteq G^{geom}(f)$.
\end{thm}

\subsection{Directions for future work}
Our results in Theorem~\ref{main-theorem}, \ref{item-MT-1}, lead to an interesting question about the invariable generation of profinite geometric iterated monodromy groups of PCF rational maps. It is straightforward to see that a finite invariable generating set (if it exists) of the profinite geometric iterated monodromy group associated to a PCF cubic polynomial with two distinct finite critical points must contain an element that acts transitively on the first level of the ternary rooted tree $T$; compare Example~\ref{ex-counterexample}. In particular, we can always use any odometer from the group (e.g., the standard generator corresponding to the point at $\infty$) as such an element; compare Theorem~\ref{main-theorem}, \ref{item-MT-1} and Theorem~\ref{thm-conjugates},\ref{invariable-generating-set}.
On the other hand, Ejder \cite{Ejder2024} showed that there exist quadratic rational maps whose profinite arithmetic iterated monodromy groups do not contain an odometer. Nevertheless, it follows from \cite[Lem.~1.3.2]{Pink2013} that the geometric iterated monodromy groups of PCF quadratic rational maps are finitely invariably generated.

\begin{prob}
{\rm
   Are the profinite arithmetic and geometric iterated monodromy groups associated to PCF rational maps in degree $d \geq 3$ always finitely invariably generated?
    }
\end{prob}

\begin{prob}
{\rm
    Let $f$ be a PCF rational map of degree $d \geq 2$. Does the ramification portrait of $f$ determine the associated profinite geometric iterated monodromy group (up to conjugation)?
}
\end{prob}

It would also be interesting to study further, more subtle, properties of conjugacy in profinite arithmetic and geometric iterated monodromy groups, such as the local-global conjugacy questions addressed in the work of Goksel \cite{Goksel2024} for quadratic polynomials.

\subsection{Organization of the paper}
 The rest of the paper is organized as follows. In Section~\ref{recursive}, we review the basic techniques and results about conjugacy in the automorphism group of the $d$-ary rooted tree,  which we extensively use further in the paper. In Section \ref{sec: descrete-img}, we review the theory of discrete iterated monodromy groups, and use it to construct standard topological generating sets for profinite geometric iterated monodromy groups. This leads to a definition of \ref{main-assumption}-restricted model groups in Section~\ref{Section:Model gps}, where we also study their basic properties. The main result of the paper, i.e., Theorem~\ref{main-theorem}, is proved in Section~\ref{sec: inv-generation}. We prove Theorem~\ref{thm-reg-branch} and Corollary~\ref{polynomial-branch} in Section~\ref{branch-torsion}, and Theorem~\ref{thm-filtration}~in Section \ref{section-filtration}.

\section{Actions on trees}\label{recursive}

In this section, we recall the necessary background on the automorphism groups of regular rooted trees. We refer the reader to \cite{Nekrashevych2005} for more on this subject.

\subsection{Tree of words and its automorphism group}\label{tree-of-words} Throughout this section, we fix an alphabet $\sX$ of $d\geq 2$ letters.   We denote by $\sX^*:=\bigsqcup_{n\geq 0} \sX^n$ the set of all finite
words in the alphabet $\sX$. As usual, given two words $v,u\in \sX^*$ we write $vu$ for the concatenation of $v$ and~$u$. The set $\sX^*$ has a natural structure
of a $d$-ary rooted tree, see Definition~\ref{defn-d-ary}: we define the root to be the empty word $\emptyset$ and connect every word $v\in \sX^n$ to all words of the form
$vx\in \sX^{n+1}$ for an arbitrary letter $x\in \sX$ and each $n\geq 0$. 
The set $\sX^*$ viewed as a rooted tree is called the \emph{tree of words} in the alphabet
$\sX$ and is denoted by $\sT_\sX$, or simply $\sT$ when the alphabet $\sX$ is clear from the context. In this paper, we often have $\sX = \{1,\ldots,d\}$, and in the proofs of our main results, we focus on the case $\sX = \{1,2,3\}$.

For $n\geq 0$, the \emph{$n$-th level} of the tree $\sT$ is given by all words in $\sX^n$. We set $\sT_n$ to be the finite subtree of $\sT$ composed of all vertices at levels $\leq n$. Moreover, given $v\in\sX^*$, we denote by $\sT_{(v)}$ the subtree of $\sT$ with the vertex set $(v):=\{vu: u\in \sX^*\}$, that is, $\sT_{(v)}$ is the subtree of $\sT$ rooted at $v$. Clearly, $\sT_{(v)}$ is isomorphic to the whole tree of words $\sT$ via the shift $\shift_v\colon  \sT_{(v)} \to \sT$ defined by $vu \mapsto u$ for $u\in \sX^*$.

We use the notation $\AT:=\Aut(\sT)$ for the automorphism group of the tree $\sT$, that is, the group of all bijective maps $g\colon \sT\to \sT$ that preserve the adjacency of the vertices of $\sT$. In this paper, we consider the \emph{right action} of $\AT$ on the tree $\sT$. So, the image of a vertex $v$ under the action of an element $g\in \AT$ is denoted by $v\cdot g$, and in the product $gh$ the element $g$ acts first. Similarly, for each $n\geq 0$, we denote by $\AT_n:=\Aut(\sT_n)$ the automorphism group of the finite tree $\sT_n$, and we consider its right action on $\sT_n$. We write $\Pi_n\colon \AT\to \AT_n$ for the natural projection map induced by restricting the action of $g\in \AT$ to $\sT_n$, i.e., $\Pi_n(g) = g|_{\sT_n}$. 

The automorphism group $\AT$ has a natural profinite structure
\[\AT=\varprojlim \AT_n,\]
that is, it is the inverse limit of discrete finite groups $\AT_n$ with respect to the canonical projection maps $\pi_{m,n}\colon \AT_m\to \AT_n$ (for $m\geq n\geq 0$) defined by $g\mapsto g|_{\sT_n}$. It is straightforward to check that the profinite topology on $\AT$ agrees with the uniform topology induced by the following metric on $\AT$: given two distinct elements $g,h\in \AT$, we set $\dist(g,h)=\frac{1}{2^n}$, where $n\geq 0$ is the largest level of $\sT$ on which $g$ and $h$ agree; that is, $g|_{\sT_n}=h|_{\sT_n}$ and $g|_{\sT_{n+1}} \neq h|_{\sT_{n+1}}$. In particular, $\AT$ is complete under the profinite (or uniform) topology.

\subsection{Wreath product structure of \texorpdfstring{$\Aut(\sT_{\sX})$}{Aut(T\_X)}}

The automorphism group $\AT=\Aut(\sT_\sX)$ is canonically isomorphic to the permutational wreath product $\AT\wr\operatorname{Sym}(\sX)$. To describe this isomorphism, we assume, for simplicity, that $\sX=\{1,\dots,d\}$, so that $\operatorname{Sym}(\sX)=\Sym_d$.

First, we recall that the permutational wreath product $\AT\wr\Sym_d$ is the semidirect product $\AT^d \rtimes \Sym_d$, where an element $\sigma\in\Sym_d$ acts on $\AT^d$ by permuting the factors according to its action on $\sX$, that is, $$(g_1, g_2, \dots, g_{d})\mapsto (g_{1\cdot \sigma}, g_{2\cdot \sigma}, \dots,g_{{d}\cdot \sigma}).$$
In particular, the multiplication rule for elements of $\AT\wr\Sym_d$ is given by the following formula:
\[
\displaystyle \left(g_1, g_2, \dots, g_d\right)\sigma_g \,\, \left(h_1, h_2, \dots, h_d\right)\sigma_h = \left( g_{1}\, h_{1\cdot {\sigma_g}}, g_{2}\, h_{2\cdot{\sigma_g}}, \dots, g_{d}\, h_{{d}\cdot{\sigma_g}}\right) \sigma_g \sigma_h.
\]

For every $g\in \AT$ and $v\in \sX^*$, we define an automorphism $g|_v\colon \sT\to \sT$, called the \emph{section} (or \emph{restriction}) of $g$ at the vertex $v$, by 
\[g|_v := \shift_{v\cdot{g}} \circ g \circ \shift^{-1}_{v}.\]
It is then immediate that the following identities hold for all $v,u \in \sX^*$ and $g,h\in \AT$:
\[g|_{vu} = (g|_{v})|_{u} \quad \text{ and } \quad  (gh)|_v = (g|_v)(h|_{v\cdot{g}}).\]
In particular, $g^{-1}|_x = (g|_{x \cdot g^{-1}})^{-1}$ for $g \in W$ and $x \in \sX$.
In a similar way, for each $g\in \AT_n$ and $0\leq k\leq n$ we introduce the sections $g|_v\in \AT_{n-k}$ of $g$ at the vertices $v\in \sX^k$.

Let us now consider the natural map
\begin{equation}\label{eq-can-isom}
    \Psi\colon \AT\to \AT\wr\Sym_d
\end{equation}
given by 
\begin{equation*}
    g\mapsto \left(g|_1,g|_2,\dots, g|_d\right) \, \sigma_g,
\end{equation*}
where $\sigma_g$ denotes the permutation in $\Sym_d$ providing the action of $g$ on the first level of $\sT$. It is straightforward to verify that $\Psi$ is an isomorphism. We identify an element $g\in \AT$ with its image $\Psi(g)\in \AT\wr\Sym_d$, called the \emph{wreath recursion} of $g$. 

Similarly, for every $n\geq 1$ we have that $\AT_n$ is canonically isomorphic to $\AT_{n-1}\wr \Sym_d$. Hence, $\AT_n$ is the iterated permutational wreath product $[\Sym_d]^n :=\Sym_d\wr\dots\wr\Sym_d$ with $n$ factors, and $\AT$ is the infinite iterated permutational wreath product
\[[\Sym_d]^\infty:=\dots\Sym_d\wr\Sym_d\wr\Sym_d.\] Finally, we note that the symmetric group $\Sym_d$ admits a natural embedding into $\AT$ (and each $\AT_n$ for $n\geq 1$) defined by
\[\sigma\mapsto (\one, \, \dots, \one)\sigma,\]
so we will frequently identify $\sigma\in \Sym_d$ with the respective element of $\AT$ (and $\AT_n$).

\subsection{Subgroup closures, self-similarity, branching, and odometers in \texorpdfstring{$\Aut(\sT_{\sX})$}{Aut(T\_X)}}\label{sec: subgroups-in-W} 

Let $G$ be a subgroup of $\AT=\Aut(\sT_\sX)$. We denote by $\overline{G}$ the closure of $G$ in $\AT$. Furthermore, for each $n\geq 0$ we denote by $G_n$ the following subgroup of $\AT_n$:
\[G_n:=\{g|_{\sT_n}:\, g\in G\}.\]
It is immediate that $\overline{G}$ is a profinite subgroup of $\AT$ with $\overline{G}=\varprojlim G_n$, where the inverse system is given by the natural projection maps $\pi_{n,k}\colon G_n\to G_k$ for $0\leq k\leq n$, which are the restrictions of the maps $\pi_{n,k}\colon \AT_n \to \AT_k$ considered in Section~\ref{tree-of-words}. We also recall the following standard result. 

\begin{lemma}\label{lem-same-closures}
    The closures $\overline{G}, \overline{H}$ of two subgroups $G,H\subset \AT$ coincide if and only if $G_n=H_n$ for all $n\geq 0$.
\end{lemma}

For $g_1,\dots, g_k\in W$, we denote by $\llangle g_1,\dots,g_k\rrangle$ the closure of the countable group $\langle g_1,\dots, g_k\rangle$ in $\AT$. Furthermore, we say that a subgroup $H\subset \AT$ is \emph{topologically generated by
$g_1,\dots, g_k$} if $H= \llangle g_1,\dots,g_k\rrangle$.

\begin{defn}\label{defn-self-similar-replicating}
    A subgroup $G\subset \AT$ is said to be \emph{self-similar} if for all $g\in G$ and $v\in \sX^*$ we have $g|_v\in G$. A self-similar group $G$ is called \emph{self-replicating} (or \emph{recurrent}) if it acts transitively on $\sX$ and for some (and thus for all) $x\in \sX$ the homomorphism \[\phi_x\colon \St_G(x) \to G, \quad g \mapsto g|_x,\] is surjective, where $\St_G(x):=\{g\in G: \, x\cdot g=x\}$ is the stabilizer of the vertex $x$.
\end{defn}

It follows immediately that $G$ is self-similar if and only if $g|_x\in G$ for every generator $g$ of $G$ and  $x\in \sX$. 
The following standard lemma allows one to define finitely generated self-similar subgroups of $\AT$ by providing a system of recursive formulas.

\begin{lemma}\label{EXDOC: lem: G is self-similar}
Suppose $\sX=\{1,\dots, d\}$, and consider the following system of $k\geq 1$ equations: 
\begin{align}\label{eq: rec-syst}
\begin{aligned}
g_1 &= \left( h_{1,1} \, ,\,  \dots \, , \, h_{1,d}\right) \sigma_1,\\
g_2 &= \left( h_{2,1} \, ,\,  \dots \, , \, h_{2,d}\right)  \sigma_2,\\
& \vdots \\
g_k &= \left( h_{k,1} \, ,\,  \dots \, , \, h_{k,d}\right) \sigma_k, \\
\end{aligned}
\end{align}
where each $h_{i,j}$ is a (non-empty) finite word in $\{g_1,\, g_1^{-1},\, \dots,\, g_k,\, g_k^{-1}\}\sqcup \AT$ and each $\sigma_i$ is a permutation in $\Sym_d$. Then the following statements are true:
\begin{enumerate}[label=(\roman*),font=\normalfont]
    \item\label{item: rec-syst-i} There are unique elements $g_1,\dots,g_k\in \AT$ that satisfy all the identities of \eqref{eq: rec-syst}.
    \item\label{item: rec-syst-ii} If each $h_{i,j}$ is a word in $\{\one, g_1,\, g_1^{-1},\, \dots,\, g_k,\, g_k^{-1}\}$, then the respective subgroups $G:=\langle g_1,\dots, g_k\rangle$ and $\overline{G}:=\llangle g_1,\dots,g_k\rrangle$ of $\AT$ are self-similar.
\end{enumerate}
\end{lemma}
\begin{proof}
    The proof is standard, and we leave some straightforward details to the reader. By induction on the level $n$ of the tree $\sT_\sX$, we deduce that System~\eqref{eq: rec-syst} defines uniquely the action of each $g_i$ on $\sT_n$, and thus \ref{item: rec-syst-i} follows. When every $h_{i,j}$ is a word in $\{\one, g_1,\, g_1^{-1},\, \dots,\, g_k,\, g_k^{-1}\}$, we also immediately have that $g_i|_x\in G$ for every generator $g_i$ and all $x\in \sX$. Finally,
    fix $g \in \overline{G}$ and let $\{h_m\} \subset G$ be a sequence converging to $g$, i.e., for every $n\geq 1$ there is $M_n \geq 1$ such that $h_m|_{\sT_n}  = g|_{\sT_n}$ for all $m \geq M_n$. It then follows that for each $x\in \sX$ the sequence of sections $\{h_m|_x\}$ converges to the section $g|_x$. Note that each $h_m|_x\in G$ due to the self-similarity of~$G$. Since $\overline{G}$ is closed, we conclude that $g|_x\in \overline{G}$, and thus $\overline{G}$ is self-similar as well.
\end{proof}

\begin{defn}\label{weakly-branch}
    A subgroup $G\subset \AT$ is called \emph{regular weakly branch on a subgroup $H\subset G$} if $H$ is non-trivial and $$\underbrace{H\times \dots \times  H}_{\text{$d$ factors}} \subset H;$$
if, in addition, the subgroup $H$ has finite index in $G$, we say that $G$ is \emph{regular branch on~$H$}. 
\end{defn}

Here and below, given $H, H_1, \dots, H_d \subset \AT$, the notation $H_1\times \dots \times H_d \subset H$ means that $H$ contains the preimage $\Psi^{-1}(H_1\times \dots \times H_d)$ under the canonical isomorphism \eqref{eq-can-isom}.

\begin{defn}\label{defn-odometer}
    An automorphism $g\in \AT$ is called an \emph{odometer} if $g$ acts transitively on each level of the tree $\sT_\sX$. 
\end{defn}
It follows immediately that $g$ is an odometer if and only if $g|_{\sT^n}$ has order $d^n$ for all $n\geq 0$, where $d=|\sX|$. The recursively defined element
\[
g = \left(\one, \dots, \one, g\right) (1\, 2\, \cdots \, d)\]
is an example of an odometer, called the \emph{standard odometer}.

\subsection{Conjugacy in \texorpdfstring{$\Aut(\sT_{\sX})$}{Aut(T\_X)}} In this subsection, we record some basic properties of conjugation in $\AT=\Aut(\sT_\sX)$ with $\sX=\{1,\dots, d\}$, which will be used in the rest of the paper. These properties build on the results previously established for the binary rooted tree in \cite{Pink2013}. Although a similar generalization has appeared in the recent work \cite{AdamsHyde}, we include it here for the sake of completeness. Also, note that \cite{AdamsHyde} uses the left action of $\Aut (\sT_\sX)$ on $\sT_\sX$, while we use the right action. First, we set up some notation.

Given $g,g'\in \AT$, we write $g\sim_\AT g'$ if $g$ and $g'$ are conjugate in $\AT$, that is, there exists $w\in \AT$ such that $g=w \,g'\, w^{-1}$. Similarly, if $G$ is a subgroup of $\AT$ and $g,g'\in G$, we write $g\sim_G g'$ if there exists $h\in G$ such that $g=h \,g'\, h^{-1}$. We will also use notations $\sim_{\AT_n}$ and $\sim_{G_n}$ to denote the conjugacy relations in $\AT_n$ and $G_n$, respectively, for each $n\geq0$.

Given a (possibly trivial) cycle $\tau=(x_1 \, \cdots \, x_r)\in \Sym_d$ with $r\geq 1$, we write $\supp(\tau)$ for the \emph{support of $\tau$}, that is, $\supp(\tau)= \{x_1,\dots, x_r\}$. We recall that any permutation $\sigma\in \Sym_d$ admits a \emph{disjoint cycle decomposition}, that is, $\sigma$ may be represented as the product 
$\sigma= \tau_1 \cdots \tau_s$ 
of cycles $\tau_j\in \Sym_d$ with pairwise disjoint supports whose union $\bigcup^s_{j=1} \supp(\tau_j)$ equals $\sX$. By an \emph{orbit transversal} for the action of $\sigma$ on $\sX=\{1,\dots,d\}$ we mean a subset of $\sX$ that contains precisely one element from the support of each cycle $\tau_j$. Note that the lengths of the cycles $\tau_j$ form an integer partition of $d$, called the \emph{cycle type} of $\sigma$. Furthermore, two permutations $\sigma, \sigma'\in \Sym_d$ are conjugate if and only if they have the same cycle type. Finally, we denote by $\Fix(\sigma)=\{x\in \sX:\, x\cdot \sigma = x\}$ the set of fixed points of $\sigma\in \Sym_d$.

Fix an arbitrary $g \in \AT$ and suppose $\left(g_1\,,\, \dots \, , \, g_d\right) \sigma_g$ is the associated wreath recursion. For each (possibly trivial) cycle $\tau = (x_1 \, \cdots \, x_r)$ in the disjoint cycle decomposition of $\sigma_g$, we denote by $g_\tau$ the product $g_{x_1}\cdots g_{x_r}=(g^r)|_{x_1}$ of sections of $g$ at the vertices in the support of $\tau$ taken in the order of the cycle $\tau$. We call $g_\tau$ a \emph{cyclic section product} of $g$. Clearly, the element $g_\tau$ is well defined only up to conjugacy: a different choice of the starting value $x_1$ in the cycle $\tau$ results in a conjugate product. Finally, for all $n\geq 1$ and $g=\left(g_1\,,\, \dots \, , \, g_d\right)\sigma_g\in \AT_n$, we define the element $g_\tau \in \AT_{n-1}$ for each cycle $\tau$ of $\sigma_g$ in an analogous way.

For the convenience of the reader, we record the following lemma describing the conjugacy action on wreath recursions. 
\begin{lemma}\label{lem: conj-action}
    Let $g=\left(g_1\,,\, \dots \, , \, g_d\right) \sigma_g$ and  $w=\left(w_1\,,\, \dots \, , \, w_d\right) \sigma_w$ be arbitrary elements in $\AT$. Then for the conjugate $g' := wgw^{-1}=\left(g'_1\,,\, \dots \, , \, g'_d\right) \sigma_{g'}$ we have
    \[\sigma_{g'}= \sigma_w\sigma_g \sigma_w^{-1} \quad \text{and} \quad g'_x=w_x\,g_{x\cdot \sigma_w}\,w^{-1}_{x\cdot \sigma_{g'}} \quad \text{for each $x\in \sX$.}\]
    Moreover, the analogous statement holds for all $n\geq 1$ and $g,w\in \AT_n$. 
\end{lemma}

\begin{cor}\label{cor: special-element}
    Let $G\subset \AT$ be a subgroup with $G_1 = \Sym_d$. Then, for any $g=\left(g_1\,,\, \dots \, , \, g_d\right) \sigma_g \in G$ with trivial $\sigma_g$, and any $\mu\in {\mathcal S}_d$, there exists some $h=\left(h_1\,,\, \dots \, , \, h_d\right) \sigma_h  \in G$ with $\sigma_h = \mu$ such that 
    \[(hgh^{-1})|_{x}  =(h_x\,g_{x\cdot \mu}\,h_x^{-1})\quad \text{for each $x \in \sX$}.\] In particular, if $g_x=\one$ for some $x\in \sX$, then $(hgh^{-1})|_{x\cdot\mu^{-1}}=\one$.
    Moreover, the analogous statement holds for all $n\geq 1$ and $g\in G_n$.
\end{cor}

The following technical result allows us to choose a representative $h$ of the conjugacy class of a given automorphism $g\in\AT$ so that the wreath recursion of $h$ has a convenient simple form. In particular, this representative $h$ has the maximal possible number of trivial sections on the first level; see Remark~\ref{rem: conjugacy-in-W} below.

\begin{lemma}\label{lem: conjugacy-representative}
    Let $g=\left(g_1\,,\, \dots \, , \, g_d\right) \sigma_g$ be an arbitrary element in $\AT$, and suppose $\sigma_g=\tau_1\cdots \tau_s$ is a disjoint cycle decomposition of $\sigma_g$. For each cycle $\tau_j$, choose a letter $x_j\in \supp(\tau_j)$. 

    Consider the element $h\in \AT$ defined by the wreath recursion $h=\left(h_1\,,\, \dots \, , \, h_d\right) \sigma_h$, where 
    \[\text{$\sigma_h= \sigma_g$ \quad and \quad  
    $h_x= \begin{cases} g_{\tau_j}, \quad &\text{if $x=x_j$ for some cycle $\tau_j$}\\ \one, \quad &\text{otherwise}.\end{cases}$}\]
    Then the elements $g$ and $h$ are conjugate in $\AT$.
\end{lemma}

In other words, given an arbitrary $g\in \AT$ and an orbit transversal $\{x_1,\dots, x_s\}$  for the action of $\sigma_g$ on $\sX$, we may find a conjugate automorphism $h\sim_\AT g$ such that all sections $h|_x$ are trivial for all $x\in \sX\setminus \{x_1,\dots, x_s\}$.

\begin{proof}
    The proof is by a direct computation. When $\sigma_g$ has a single non-trivial cycle in its disjoint cycle decomposition, we may assume that $g$ has the form 
    $$g=\left(g_1,\, \dots\, , \, g_d\right)(1\,2\, \cdots\,r)$$ for some $g_1,\dots,g_d\in \AT$ and $2\leq r\leq d$. We then conjugate $g$ by 
    $$w:=\left(\one\, ,\, g_2\cdots g_r\, ,\, g_3\cdots g_r\, ,\, \dots \, , \, g_r\, , \, \one \, , \one \, , \dots \, , \, \one\right) \in W$$
    to obtain that 
    $$g\, \sim_\AT \, \left(g_1\cdots g_r\, ,\, \one\, ,\, \dots\, ,\, \one\, ,\, g_{r+1}\, ,\, g_{r+2}\, ,\, \dots\, ,\, g_d\right)(1\,2\,\cdots\,r).$$
    The general case follows by induction on the number of non-trivial cycles in the disjoint cycle decomposition of $\sigma_g$.
\end{proof}

\subsubsection{Conjugacy criteria}  We now provide several criteria for conjugacy in $\AT=\Aut(\sT_\sX)$, where $\sX = \{1,\ldots,d\}$. The lemma below is well-known.

\begin{lemma} \label{lem: conjugacy-works-levelwise}
    Let $g,g'\in \AT$. Then $g\sim_\AT g'$ if and only if $g|_{\sT_n}\sim_{\AT_n} g'|_{\sT_n}$ for each $n\geq 0$.
\end{lemma}

In other words, two automorphisms $g,g'\in \AT$ are conjugate in $\AT$ if and only if their restrictions up to each finite level $n$ of the tree $\sT_\sX$ are conjugate in $\AT_n$.

The next statement provides a very useful criterion for conjugacy of two elements in $\AT$ in terms of their wreath recursions.  
 
\begin{prop} \label{prop: conjugacy-in-W}
    Let $g,g'\in \AT$, and let $\sigma_g,\sigma_{g'}\in \Sym_d$ be their respective actions on $\sX$. Then the following are equivalent:
    \begin{enumerate}[label=(\roman*),font=\normalfont]
    \item\label{item: conjugacy-i} $g$ and $g'$ are conjugate in $\AT$;
    \item\label{item: conjugacy-ii} 
    There exists some $\mu\in \Sym_d$ with $\sigma_g=\mu \sigma_{g'}\mu^{-1}$ such that, for each cycle $\tau$ in the disjoint cycle decomposition of $\sigma_g$, we have $g_\tau\sim_\AT g'_{\mu^{-1}\tau\mu}$. In particular, for each $x \in {\rm Fix}(\sigma_g)$ we have $g|_x \sim_\AT g'|_{x \cdot \mu}$.
    \end{enumerate}
    Moreover, the analogous statement holds for all $n\geq 1$ and $g,g'\in \AT_n$. 
\end{prop}

In other words, two automorphisms $g,g'$ are conjugate in $\AT$ if and only if there exists a one-to-one correspondence between the cycles from the disjoint cycle decompositions of $\sigma_g$ and $\sigma_{g'}$ that respects both the cycle lengths and the cyclic section products up to conjugacy. Before we prove Proposition~\ref{prop: conjugacy-in-W}, we discuss an example and make some observations that illustrate the subtlety of the problem.

\begin{ex}
    Let $d=7$, and consider an element $g\in \AT$ of the form $$g=\left(g_1\, , \, g_2 \, , \, g_3 \, , \, g_4 \, , \, g_5 \, , \, g_6 \, , \, g_7\right)(3\, 5 \, 1)(6\, 2)(4\, 7).$$ 
    Suppose that $g'\in \AT$ with 
    $$g'=\left(g'_1\, , \, g'_2 \, , \, g'_3 \, , \, g'_4 \, , \, g'_5 \, , \, g'_6 \, , \, g'_7\right) \sigma_{g'}$$
    is conjugate to $g$ in $\AT$. Then there exists a partition
    $$\{x_1,x_2,x_3\} \sqcup \{x_4,x_5\}\sqcup\{x_6,x_7\}$$
    of $\sX$ such that 
    \begin{equation}\label{eq-ex-conj-1}
    \sigma_{g'}=(x_1\, x_2 \, x_3)(x_4\, x_5)(x_6\, x_7)
    \end{equation}
    and such that all of the following conjugacy relations hold:
    \begin{align}\label{eq-ex-conj-2}
    \begin{aligned}
        g_3g_5g_1&\sim_\AT g'_{x_1}g'_{x_2}g'_{x_3},\\
        g_6g_2&\sim_\AT g'_{x_4}g'_{x_5},\\
        g_4g_7&\sim_\AT g'_{x_6}g'_{x_7}.
    \end{aligned}
    \end{align}
    Conversely, if the wreath recursion of $g'$ satisfies \eqref{eq-ex-conj-1} and \eqref{eq-ex-conj-2} for some partition of $\sX$ as above, then we have $g\sim_\AT g'$. 
\end{ex}

\begin{warn}
    It is important to keep in mind that even when the first-level actions of $g, g' \in \AT$ coincide and the respective sections of $g$ and $g'$ are pairwise conjugate (i.e., $g|_x\sim_\AT g'|_x$ for all $x\in \sX$), we do not necessarily have that $g\sim_\AT g'$. 
\end{warn}

\begin{rmk}\label{rem: conjugacy-in-W}
Suppose we are given two conjugate automorphisms $g,g'\in \AT$ with $\mu\in \Sym_d$ as in item \ref{item: conjugacy-ii} of Proposition~\ref{prop: conjugacy-in-W}. The proposition implies that for every cycle $\tau$ in the disjoint cycle decomposition of $\sigma_g$, either $g_\tau = \one$ or at least one of the sections $g'|_x$ with $x\in \supp(\mu^{-1}\tau\mu)$ is non-trivial. In particular, the element $h\sim_\AT g$ provided by Lemma~\ref{lem: conjugacy-representative} has the maximal possible number of trivial first-level sections among the conjugates of $g$.
\end{rmk}

\begin{proof}[Proof of Proposition~\ref{prop: conjugacy-in-W}]
    We prove both implications for $g,g'\in \AT$; the statement for $g,g'\in \AT_n$ follows similarly. 
    
    \ref{item: conjugacy-i}$\Rightarrow$\ref{item: conjugacy-ii} Fix arbitrary $g=\left(g_1\,,\, \dots \, , \, g_d\right) \sigma_g, \, g'=\left(g'_1\,,\, \dots \, , \, g'_d\right) \sigma_{g'}\in \AT$, and suppose there exists $w=\left(w_1\,,\, \dots \, , \, w_d\right) \sigma_w \in \AT$ such that $g=wg'w^{-1}$. 
    By comparing the first-level actions, we immediately obtain that $\sigma_g=\sigma_w\,\sigma_{g'}\,\sigma_w^{-1}$.
    Let $\tau=(x_1 \, \cdots \,x_r)$ be a cycle in the disjoint cycle decomposition of $\sigma_g$. Then $\sigma_w^{-1}\tau \sigma_w= \big((x_1\cdot{\sigma_w}) \,\, \cdots\,\, (x_r\cdot {\sigma_w})\big)$ is a cycle in the disjoint cycle decomposition of $\sigma_{g'}$. Furthermore, using Lemma~\ref{lem: conj-action}, 
    we have   
    \begin{align*}
    g_\tau&=g_{x_1}\cdots g_{x_r}=(w_{x_1} \, g'_{x_1\cdot {\sigma_w}} \, w^{-1}_{x_1\cdot {\sigma_g}}) (w_{x_2} \, g'_{x_2\cdot {\sigma_w}} \, w^{-1}_{x_2\cdot {\sigma_g}}) \cdots (w_{x_r} \, g'_{x_r\cdot {\sigma_w}} \, w^{-1}_{x_r\cdot {\sigma_g}})\\
    &=w_{x_1} \, g'_{\sigma_w^{-1}\tau\sigma_w} \, w^{-1}_{x_1}.
    \end{align*}
    It follows that $g_\tau\sim_\AT g'_{\sigma_w^{-1}\tau\sigma_w}$, and thus \ref{item: conjugacy-ii} holds with $\mu=\sigma_w$.

    \ref{item: conjugacy-ii}$\Rightarrow$\ref{item: conjugacy-i} Let $\sigma_g=\tau_1\cdots\tau_s$ be a disjoint cycle decomposition of $\sigma_g$. Suppose there exists $\mu\in \Sym_d$ with $\sigma_g=\mu \sigma_{g'}\mu^{-1}$ and such that $g_{\tau_j}\sim_\AT g'_{\mu^{-1}{\tau_j}\mu}$ for each cycle $\tau_j$. Pick an orbit transversal $\{x_1,\dots, x_s\}$ for the action of $\sigma_g$ on $\sX$ so that $x_j\in \supp(\tau_j)$ for each $j$. Then $\{x_1\cdot{\mu},\dots, x_s\cdot{\mu}\}$ is an orbit transversal for the action of $\sigma_{g'}$ on $\sX$ with $x_j\cdot \mu\in \supp(\mu^{-1}\tau_j\mu) $. Consider the automorphisms $h=\left(h_1\,,\, \dots \, , \, h_d\right) \sigma_h, \, h'=\left(h'_1\,,\, \dots \, , \, h'_d\right) \sigma_{h'}\in \AT$ satisfying \[\text{$\sigma_h=\sigma_g$ \quad  and \quad  $\sigma_{h'}=\sigma_{g'}$},\]
    as well as
    \[\text{$h_x= \begin{cases} g_{\tau_j}, \quad &\text{if $x=x_j$ for some cycle $\tau_j$}\\ \one, \quad &\text{otherwise}\end{cases}$ \quad and \quad   $h'_x= \begin{cases} g'_{\mu^{-1}\tau_j\mu}, \quad &\text{if $x=x_j\cdot \mu$ for some cycle $\tau_j$}\\ \one, \quad &\text{otherwise}.\end{cases}$} \]
    Then $g\sim_\AT h$ and $g'\sim_\AT h'$ by Lemma~\ref{lem: conjugacy-representative}.
    
    We now show that $h\sim_\AT h'$, which will conclude the proof of \ref{item: conjugacy-i}. First, for each $j\in\{1,\dots,s\}$, we fix an element $u_{j}\in \AT$ such that $g_{\tau_j}=u_{j}\, g'_{\mu^{-1}\tau_j\mu} \,u_{j}^{-1}$. Consider the automorphism $w=\left(w_1\,,\, \dots \, , \, w_d\right) \sigma_w$ with $\sigma_w = \mu$ and with $w_x=u_j$ whenever $x\in \supp(\tau_j)$. We claim that 
    \begin{equation}\label{eq: conjugacy-in-W-2}
    h=wh'w^{-1}.    
    \end{equation}
    Indeed, by construction we have $$\sigma_h=\sigma_g=\mu\sigma_{g'}\mu^{-1}=\sigma_w\sigma_{h'}\sigma_w^{-1},$$
    and thus the first-level actions of both sides in \eqref{eq: conjugacy-in-W-2} agree. At the same time, for each cycle $\tau_j$ and $x\in \supp(\tau_j)$ we have 
    \[x\cdot (wh'w^{-1})= x\cdot h = x\cdot \sigma_h = x\cdot \tau_j \in \supp(\tau_j),\]
    and hence
    \begin{align*}
    (wh'w^{-1})|_x&=w|_x\, h'|_{x\cdot w} \, w^{-1}|_{x\cdot(wh')}=u_j\, h'|_{x\cdot \mu} \,u^{-1}_j\\
    &=\begin{cases}
    u_j\,g'_{\mu^{-1}\tau_j\mu}\,u_j^{-1}= g_{\tau_j}, \quad &\text{if $x=x_j$}\\ u_j \, \one \, u_j^{-1}=\one, \quad &\text{otherwise}.\\
    \end{cases}
    \end{align*}
    It follows that all first-level sections of $h$ and $wh'w^{-1}$ agree, and therefore the identity~\eqref{eq: conjugacy-in-W-2} holds, finishing the proof.
\end{proof}

The above proof of the implication \ref{item: conjugacy-i}$\Rightarrow$\ref{item: conjugacy-ii} yields the following statement regarding conjugacy within a self-similar subgroup of $\AT$.  

\begin{cor} \label{cor: conjugacy-in-G}
    Let $G\subset \AT$ be a self-similar subgroup. Suppose $g,g'\in G$ are conjugate in $G$, and let $\sigma_g,\sigma_{g'}\in \Sym_d$ be their respective actions on $\sX$. 
    Then  
    there exists $\mu\in \Sym_d$ with $\sigma_g=\mu \sigma_{g'}\mu^{-1}$ such that for each cycle $\tau$ in the disjoint cycle decomposition of $\sigma_g$ we have $g_\tau\sim_G g'_{\mu^{-1}\tau\mu}$. In particular, for each $x \in {\rm Fix}(\sigma_g)$ we have $g|_x \sim_G g'|_{x \cdot \mu}$.

  Moreover, the analogous statement holds for all $n\geq 1$ and $g,g'\in G_n$ with $g\sim_{G_n} g'$.
\end{cor}

\subsubsection{Systems of conjugacies in \texorpdfstring{$\Aut(\sT_{\sX})$}{Aut(T\_X)}}  In the next sections, when working with the discrete (or profinite geometric) iterated monodromy groups, we will be able to describe the wreath recursion structure of their (topological) generating sets only up to conjugacy. More formally, we will be considering elements $g_1,\dots, g_k\in \AT=\Aut(\sT_\sX)$ that satisfy the following system of recursive relations:
\begin{align}\label{eq: rec-syst-conj}
\begin{aligned}
g_1 &\sim_\AT \left( h_{1,1} \, ,\,  \dots \, , \, h_{1,d}\right) \sigma_1,\\
g_2 &\sim_\AT \left( h_{2,1} \, ,\,  \dots \, , \, h_{2,d}\right)  \sigma_2,\\
& \vdots \\
g_k &\sim_\AT  \left( h_{k,1} \, ,\,  \dots \, , \, h_{k,d}\right) \sigma_k, \\
\end{aligned}
\end{align}
where each $h_{i,j}$ is a (non-empty) finite word in $\{\one, g_1,\, g_1^{-1},\, \dots,\, g_k,\, g_k^{-1}\}$ and each $\sigma_i$ is a permutation in $\Sym_d$. The following result shows that, under some extra assumptions on the wreath recursions on the right hand side (see Condition~\ref{cond: promoting-recursion-conjugacy} below), System~\eqref{eq: rec-syst-conj} specifies each element $g_1,\dots, g_k$ in a unique way up to conjugacy in $\AT$.  To simplify the notation, we set \[\AT[g_1,\dots,g_k]:=\AT*\mathbb{F} (g_1,\dots,g_k)\] to be the free product of $\AT$ and the free group $\mathbb{F} (g_1,\dots,g_k)$ on the set $\{g_1,\dots,g_k\}$; that is, $\AT[g_1,\dots,g_k]$ is the group formed by freely adjoining all the $g_i$'s to $\AT$.

 \begin{prop}\label{prop: promoting-recursion-conjugacy}

 Fix $k\geq 1$, and consider $k$ symbols $g_1,\dots, g_k$, as well as
 $kd$ words $h_{i,j}\in \AT[g_1,\dots, g_k]$ with $i\in\{1,\dots,k\}$ and $j\in\{1,\dots,d\}$, together with $k$ permutations $\sigma_1,\dots,\sigma_k\in \Sym_d$.

 Let us suppose that 
\begin{enumerate}[label=\normalfont{($\square$)}]
\item\label{cond: promoting-recursion-conjugacy}
     for each $i\in\{1,\dots, k\}$ and each cycle $\tau=(x_1\, \cdots \, x_r)$ in the disjoint cycle decomposition of $\sigma_i$, the product $h_{i,x_1}\cdots h_{i,x_r}$ is conjugate to $g_j^m$ in $\AT[g_1,\dots,g_k]$ for some $j\in \{1,\dots,k\}$ and $m\in \Z$ (where both $j$ and $m$ may depend on $i$ and $\tau$).
 \end{enumerate}

 Next, let $\widehat{g}_1,\dots,\widehat{g}_k\in \AT$ be $k$ automorphisms defined by the following system of recursive formulas
 \[\widehat{g}_i=\left( \widehat{h}_{i,1} \, ,\,  \dots \, , \, \widehat{h}_{i,d}\right)  \sigma_i, \quad {i=1,\dots,k},\]
 where each word $\widehat h_{i,j}$ is obtained from $h_{i,j}$ by replacing every occurrence of $g_t$ and $g_t^{-1}$ with $\widehat g_t$ and $\widehat g_t^{-1}$, respectively, for each $t=1,\dots,k$.  
 
 Finally, fix elements $\phi(g_1),\dots,\phi(g_k)\in \AT$, and consider the induced canonical quotient map $\phi\colon\AT[g_1,\dots,g_k]\to \AT$. Then the following statements are equivalent:
     \begin{enumerate}[label=(\roman*),font=\normalfont]
     \item\label{item: promoting-recursion-conjugacy-i} For every $i=1,\dots,k$ we have $\phi(g_i) \sim_\AT \left( 
     \phi(h_{i,1}) \, ,\,  \dots \, , \, \phi(h_{i,d})\right)  \sigma_i$.
     \item\label{item: promoting-recursion-conjugacy-ii} For every $i=1,\dots,k$ we have $\phi(g_i) \sim_\AT \widehat{g}_i$. 
     \end{enumerate}
 \end{prop}

\begin{warn}
We note that Proposition~\ref{prop: promoting-recursion-conjugacy} does not hold if Condition~\ref{cond: promoting-recursion-conjugacy} is dropped. Indeed, let us consider the following system of recursive relations in $\AT=\Aut(\sT_\sX)$ with $\sX=\{1,2\}$:
\begin{align}\label{ex: rec-syst-conj}
\begin{aligned}
 a &\sim_\AT \left( a \, , \,  b\, \right) (1\, 2),\\
b &\sim_\AT \left( \one \, , \, \one\right) (1\, 2).\\
\end{aligned}
\end{align}
Clearly, the unique automorphisms $\widehat a,\, \widehat b\in \AT$ with \[\text{$\widehat a = \left( \widehat a\, ,\,  \widehat b\,\right) (1\, 2)$ \quad  and \quad $\widehat b = \left( \one \, , \, \one\right) (1\, 2)$}\] satisfy~\eqref{ex: rec-syst-conj}. However, even though the elements \[\text{$a':=\left(\, \widehat b \,\widehat a\, ,\,  \one \right) (1\, 2)$ \quad and \quad $b':=\widehat b$}\] satisfy $a'\sim_\AT \widehat a$ and $b'\sim_\AT \widehat b$, they do not satisfy System~\eqref{ex: rec-syst-conj}, as can be easily verified using Proposition~\ref{prop: conjugacy-in-W}. Conversely, $a := \widehat b$ and  $b := \widehat b$ satisfy~\eqref{ex: rec-syst-conj}, but $a$ is not conjugate to $\widehat a$.
\end{warn}

\begin{proof}[Proof of Proposition~\ref{prop: promoting-recursion-conjugacy}] We will use the conjugacy criterion from Proposition~\ref{prop: conjugacy-in-W} to prove both implications.  

    \ref{item: promoting-recursion-conjugacy-i}$\Rightarrow$    \ref{item: promoting-recursion-conjugacy-ii} By Lemma~\ref{lem: conjugacy-works-levelwise}, it suffices to show the following assertion for all $n\geq0$:

    \begin{enumerate}[label=\normalfont{$(*_n)$}]
    \item\label{induction: promoting-recursion-conjugacy}
        For each $i=1,\dots, k$ we have $\phi(g_i)|_{\sT_n}\sim_{\AT_n} \widehat{g}_i|_{\sT_n}$.
    \end{enumerate}

    We use induction on $n$. The statement is trivial for $n=0$, so assume we have proved \ref{induction: promoting-recursion-conjugacy} for some $n\geq 0$. Fix an arbitrary $i\in \{1,\dots, k\}$, and consider the automorphism \[h_i:=\left( 
     \phi(h_{i,1}) \, ,\,  \dots \, , \, \phi(h_{i,d})\right)  \sigma_i\in \AT.\] Then $\phi(g_i)\sim_\AT h_i$ by Assumption~\ref{item: promoting-recursion-conjugacy-i}. Let $\tau=(x_1\, \cdots \, x_r)$ be a cycle in the disjoint cycle decomposition of $\sigma_i$. By Condition~\ref{cond: promoting-recursion-conjugacy}, there exist some $j\in \{1,\dots, k\}$ and $m\in \Z$ such that $h_{i,x_1}\cdots h_{i,x_r}$ is conjugate to $g_j^m$ in $\AT[g_1,\dots,g_k]$, which immediately implies that $\phi(h_{i,x_1})\cdots \phi(h_{i,x_r})=\phi(h_{i,x_1}\cdots h_{i,x_r})$ is conjugate to $\phi(g_j)^m=\phi(g_j^m)$ in $\AT$. It also follows that $\widehat{h}_{i,x_1}\cdots \widehat{h}_{i,x_r}$ is conjugate to ${\widehat{g}_j}^{\,m}$ in $\AT$. Using our induction hypothesis \ref{induction: promoting-recursion-conjugacy}, we deduce that
    \[\big(\phi(h_{i,x_1})\cdots \phi(h_{i,x_r})\big)|_{\sT_{n}} \,\sim_{\AT_{n}}\, \phi(g_j)^m|_{\sT_{n}}\,\sim_{\AT_{n}}\,{\widehat{g}_j}^{\,m}|_{\sT_{n}}\,\sim_{\AT_{n}}\,\big(\widehat{h}_{i,x_1}\cdots \widehat{h}_{i,x_r}\big)|_{\sT_{n}}.\]
    We conclude that $(h_i)_{\tau}|_{\sT_{n}}\sim_{\AT_{n}}(\widehat{g}_i)_{\tau}|_{\sT_{n}}$ for all cycles $\tau$ in the decomposition of $\sigma_i$.  Proposition~\ref{prop: conjugacy-in-W} then implies that $h_i|_{\sT_{n+1}}\sim_{\AT_{n+1}} \widehat{g}_i|_{\sT_{n+1}}$, and thus $\phi(g_i)|_{\sT_{n+1}}\sim_{\AT_{n+1}} \widehat{g}_i|_{\sT_{n+1}}$ as well, completing the induction step.

    \ref{item: promoting-recursion-conjugacy-ii}$\Rightarrow$    \ref{item: promoting-recursion-conjugacy-i} Let us again fix some $i\in \{1,\dots,k\}$ and consider the automorphism $h_i\in \AT$ defined as above. As before, Condition~\ref{cond: promoting-recursion-conjugacy} implies that for each cycle $\tau$ in the disjoint cycle decomposition of $\sigma_i$ there exist $j\in \{1,\dots, k\}$ and $m\in \Z$ such that $(h_i)_\tau\sim_\AT \phi(g_j)^m$ and $(\widehat g_i)_\tau\sim_\AT \widehat{g}_j^{\,m}$. Using our Assumption~\ref{item: promoting-recursion-conjugacy-ii}, we then obtain:
    \[(h_i)_\tau \,\sim_\AT\, \phi(g_j)^m \,\sim_\AT\, \widehat{g}_j^{\,m} \,\sim_\AT\, (\widehat{g}_i)_\tau.\] Since this holds for all cycles $\tau$ in the decomposition of $\sigma_i$, Proposition~\ref{prop: conjugacy-in-W} implies that $\widehat g_i\sim_\AT h_i$, and thus  $\phi(g_i)\sim_\AT h_i$ as well.  Hence, \ref{item: promoting-recursion-conjugacy-i} follows. 
\end{proof}

\section{Discrete iterated monodromy groups and their standard generators}\label{sec: descrete-img}

In this section, we briefly review the general theory of discrete iterated monodromy groups developed by Nekrashevych \cite{Nekrashevych2005}. In particular, we describe the wreath recursion structure of specific generating sets for these groups in Proposition~\ref{prop: img-gen-structure}. This will allow us to describe, up to element-wise conjugation, the recursive action of the \emph{standard} (topological) \emph{generators} of profinite geometric iterated monodromy groups in our theorems, due to Proposition~\ref{prop: discrete-vs-geometric-img} below.

\subsection{Discrete iterated monodromy group}\label{subsec: img-discr}
Throughout this section, we suppose that $f\colon \rs\to \rs$ is a PCF rational map over $\C$ of degree $d\geq 2$. Recall that $\crit(f)$ and $\post(f)=\bigcup_{n\geq 1} f^n\big(\crit(f)\big)$ denote the (finite) sets of critical and postcritical points of $f$, respectively. We also note that the map 
\begin{equation}\label{eq: fn-cover}
    f^n\colon \rs\setminus f^{-n}\big(\post(f)\big) \to \rs\setminus \post(f)
\end{equation}    
is a covering for each $n\geq 0$.

Fix an arbitrary basepoint $z_0\in \MM:=\rs\setminus \post(f)$, and consider the \emph{full backward orbit} of~$z_0$, that is, the formal
disjoint union 
$$\bT_f := \bigsqcup_{n=0}^\infty f^{-n}(z_0),$$
where we set $f^{-n}(z_0):= f^{-n}(\{z_0\})$ for simplicity. The set $\bT_f$ has a natural structure of a $d$-ary rooted tree (see Definition~\ref{defn-d-ary}): we define the root to be the basepoint $z_0$ and connect every point $z\in f^{-n}(z_0)$ with $n\geq 1$ to its image $f(z)\in f^{-(n-1)}(z_0)$ under $f$. The set $\bT_f$, viewed as a rooted tree, is called the \emph{dynamical preimage tree} of $f$ (at the basepoint~$z_0$). Note that the $n$-th level of this tree is given exactly by $f^{-n}(z_0)$.

For each $n\geq 0$, the fundamental group $\pi_1(\MM, z_0)$ acts on the $n$-th level of $\bT_f$ by the
\emph{monodromy action} for the covering \eqref{eq: fn-cover}. Namely, the image of a point $z\in f^{-n}(z_0)$ under the action of $[\gamma]\in  \pi_1(\MM, z_0)$ is defined to be the endpoint $z\cdot [\gamma]\in  f^{-n}(z_0)$ of the unique $f^n$-lift of the loop $\gamma$ that starts at $z$. The induced action of $\pi_1(\MM,z_0)$ on $\bT_f$ is called the \emph{iterated monodromy action}. It is straightforward to verify that this action preserves the rooted tree structure of $\bT_f$, that is, we get a group homomorphism
\[\Phi_f\colon \pi_1(\MM,z_0)\to \Aut(\bT_f).\]
Furthermore, up to conjugation, the iterated monodromy action does not depend on the choice of the basepoint $z_0$.

\begin{defn}
    The \emph{discrete iterated monodromy group} of a PCF rational map  $f\colon \rs\to \rs$ is defined to be 
    \[\IMG(f):=\pi_1(\MM, z_0)\, / \, \Ker(\Phi_f)\,  \cong \, \operatorname{Im}(\Phi_f). \] 
\end{defn}

We note that 
\[\Ker(\Phi_f) = \bigcap_{n=0}^\infty\Ker_n,\]
where $\Ker_n\subset \pi_1(\MM,z_0)$ denotes the kernel of the monodromy action for the covering~\eqref{eq: fn-cover}.  
It is well-known that the monodromy group $\pi_1(\MM,z_0)/\Ker_n$ of \eqref{eq: fn-cover} is isomorphic to the Galois group $\operatorname{Gal}\left(K_n/\C(t)\right)$, where $K_n$ is the finite degree extension of $\C(t)$ obtained by adjoining all the solutions of $f^n(z)=t$ to $\C(t)$ in some algebraic closure of $\C(t)$ \cite[Thm.~8.12]{Forster}. Moreover, the corresponding actions on $f^{-n}(z_0)$ and on the $d^n$ solutions of $f^n(z)=t$ are conjugate. 
Since the $d$-ary rooted tree $T$, whose vertices are identified with the solutions of the equations $f^n(z) - t=0$ as in Section~\ref{subsec: profinite-imgs}, is isomorphic to the dynamical preimage tree~$\bT_f$, it follows that the profinite geometric iterated monodromy group $G^{geom}(f) \subset \Aut(T)$ is naturally isomorphic to the profinite completion of the discrete iterated monodromy group $\IMG(f)$ in $\Aut(\bT_f)$.

\begin{prop}[{\cite[Prop.~2.1]{Jones}; see also \cite[Prop.~6.4.2]{Nekrashevych2005}}]\label{prop: discrete-vs-geometric-img}
    Let $f\in \C(z)$ be a PCF rational map. Suppose $$\overline{\IMG}(f):=\varprojlim \pi_1(\MM, z_0)/\Ker_n$$ is the profinite completion of $\pi_1(\MM, z_0)$ with respect to the sequence $\{\Ker_n\}$ of kernels of the monodromy actions for the iterated coverings \eqref{eq: fn-cover}. Then $\overline{\IMG}(f)$ is isomorphic to the profinite geometric iterated monodromy group $G^{geom}(f)$, that is, to the Galois group of $\mathcal{K}=\bigcup_{n\geq 1}{K_n}$ over $\C(t)$. Moreover, the corresponding actions on the dynamical preimage tree $\bT_f$ and on the $d$-ary rooted tree $T$ are conjugate.
\end{prop}

\subsection{Self-similarity of the iterated monodromy action}\label{subsec: self-similarity} Let $\sX$ be any alphabet of $d=\deg(f)$ letters. Clearly, the tree of words $\sT_\sX$ and the dynamical preimage tree $\bT_f$ are isomorphic, as they are both $d$-ary rooted trees. In fact, there exists a particularly natural class of such isomorphisms that are constructed as follows.

Fix an arbitrary bijection $\Lambda\colon \sX \to f^{-1}(z_0)$ between 
the vertices on the first levels of the trees $\sT_\sX$ and $\bT_f$. In addition, for each $x\in \sX$ choose a \emph{connecting path} $\ell_x\subset \MM$ from the basepoint $z_0$ to its preimage $\Lambda(x)\in f^{-1}(z_0)$. We may now inductively extend the bijection $\Lambda\colon \sX \to f^{-1}(z_0)$ to an isomorphism $\Lambda\colon \sT_\sX  \to \bT_f$ of rooted trees as follows: Set $\Lambda(\emptyset)=z_0$, and assume that $\Lambda\colon \sX^n \to f^{-n}(z_0)$ has been defined for some $n\geq 0$. Then for all $v\in \sX^n$ and $x\in \sX$, we declare $\Lambda(xv)$ to be the endpoint of the unique $f^n$-lift of the path $\ell_x$ starting at $\Lambda(v)\in f^{-n}(z_0)$. One can check that the resulting map $\Lambda\colon \sT_\sX \to \bT_f$, which we call a \emph{topological labeling map}, is indeed an isomorphism of rooted trees; see \cite[Prop.~5.2.1]{Nekrashevych2005}.

Let us conjugate the iterated monodromy action of $\pi_1(\MM,z_0)$ on $\bT_f$ by the isomorphism $\Lambda\colon  \sT_\sX \to \bT_f$. The resulting action of $\pi_1(\MM,z_0)$ (and of $\IMG(f)$) on the tree of words $\sT_\sX$ is called a \emph{standard action} of $\pi_1(\MM,z_0)$ (respectively, of $\IMG(f)$) on $\sT_\sX$. The following proposition allows one to compute this action (i.e., the wreath recursion) for a given element of the fundamental group $\pi_1(\MM,z_0)$.

\begin{prop}[{\cite[Prop.~5.2.1]{Nekrashevych2005}}]\label{prop: img-self-similar}
The standard action of $[\gamma]\in \pi_1(\MM, z_0)$ on  $\sT_\sX$ satisfies 
\begin{equation}\label{eq: img-restriction}
(xv)\cdot {[\gamma]}=y(v\cdot {[\ell_x \gamma_x \ell_y^{-1}]}),
\end{equation}
where $x\in \sX$ and $v\in \sX^*$ are arbitrary, $\gamma_x$ is the $f$-lift of the loop $\gamma$ starting at $\Lambda(x)$, and $y\in\sX$ is such that $\Lambda(y)$ is the endpoint of $\gamma_x$ (i.e.,  $y = x\cdot {[\gamma]}$).

In particular,
\begin{equation}\label{eq: img}
    \{\Lambda^{-1} \circ \Phi_f([\gamma]) \circ \Lambda : \, [\gamma]\in \pi_1(\MM, z_0)\}
\end{equation}
is a self-similar subgroup of $\Aut(\sT_\sX)$. 
\end{prop}

\subsection{\texorpdfstring{Generators of $\IMG(f)$ and $G^{geom}(f)$}{Generators of G\_discr(f) and G\_ geom(f)}}\label{subsec-generators}

In the following, we fix a topological labeling map $\Lambda\colon \sT_\sX\to \bT_f$ with $\sX=\{1,\dots,d\}$, and identify $\IMG(f)$ with the subgroup \eqref{eq: img} of $\AT=\Aut(\sT_\sX)$ induced by the associated standard action. By Proposition~\ref{prop: discrete-vs-geometric-img}, we may then also identify the profinite geometric iterated monodromy group $G^{geom}(f)$ with the closure of $\IMG(f)$ in the automorphism group $\AT$. Hence, we may transfer the information about generators of $\IMG(f)$ to $G^{geom}(f)$. Our goal in this subsection is to describe the wreath recursion structure of some natural generating set for the former group.

For each $p\in \post(f)$, let us fix a small positively oriented circle $\alpha_p$ around $p$ and connect it to the basepoint $z_0$ to form a loop $\gamma_p\subset \MM$ based at $z_0$. More formally, we pick a point $z_p\in \alpha_p$ and  connect $z_0$ to $z_p$ by a path $\beta_p\subset \rs\setminus \bigcup_{q\in P(f)}\Delta_q$, where $\Delta_q$ denotes the connected component of $\rs\setminus \alpha_q$ that contains $q\in P(f)$. Then, we define the loop $\gamma_p$ as the concatenation $\beta_p\alpha_p\beta_p^{-1}$, where we regard the positively oriented circle $\alpha_p$ as a loop based at $z_p$. We call the element $[\gamma_p]\in \pi_1(\MM, z_0)$, as well as the loop $\gamma_p$, a \emph{$p$-petal} (in~$\MM$), and denote by $g_p$ the element of $\IMG(f)$ corresponding to it. Note that a different choice of the path $\beta_p$ results in a conjugate element in $\pi_1(\MM, z_0)$. Furthermore, if the loops $\gamma_p$ (or equivalently, the paths $\beta_p$) intersect pairwise only at the basepoint $z_0$, then the set $\{[\gamma_p]:\, p\in \post(f)\}$ generates $\pi_1(\MM, z_0)$, and thus the corresponding set $\{g_p: \, p\in \post(f)\}$ generates $\IMG(f)$ and also topologically generates $G^{geom}(f)$. 

\begin{defn}\label{def: standard-generators}
    Given a collection $\{\gamma_p: p\in \post(f)\}$ of $p$-petals in $\MM=\rs\setminus \post(f)$, such that the loops $\gamma_p$ intersect pairwise only at the basepoint $z_0$, we call the set $\{[\gamma_p]:\, p\in \post(f)\}$ a \emph{standard generating set} for 
    $\pi_1(\MM, z_0)$. Furthermore, we refer to the corresponding set $\{g_p: p\in \post(f)\}$ of automorphisms in $\AT$, induced by the iterated monodromy action of $[\gamma_p]$'s on the tree $\bT_f$, as a \emph{standard generating set} for $\IMG(f)$ and $G^{geom}(f)$.
\end{defn}

The following proposition describes the (conjugacy) structure of the wreath recursions for the chosen generators.

\begin{prop}\label{prop: img-gen-structure}
    For each $p\in \post(f)$, let $[\gamma_p]\in \pi_1(\MM, z_0)$ be a $p$-petal and $g_p$ be the corresponding element of $\IMG(f)$. Fix $p\in \post(f)$, and suppose $f^{-1}(p)=\{q_1, \dots, q_s\}$, with $d_j := \deg(f,q_j)$ denoting the ramification index of $f$ at each $q_j$. Then the following statements hold: 
\begin{enumerate}[label=(\roman*),font=\normalfont]
    \item\label{item: img-gen-i} If $\sigma_{p}\in \Sym_d$ is the permutation representing the action of $g_p$ on the first level of $\sT_\sX$ and $\sigma_{p} =  \tau_1\cdots\tau_s$ is the disjoint cycle decomposition of $\sigma_{p}$, then (up to relabeling of the cycles) we have $|\supp(\tau_j)| = d_j$ for each $j=1,\dots, s$.
    \item\label{item: img-gen-ii} Suppose $\{x_1,\dots,x_s\}$ is an orbit transversal for the action of $\sigma_{p}$ on $\sX$ with each $x_j\in\supp(\tau_j)$, and consider the element $h\in \AT$ defined by the wreath recursion $h=\left(h_1\,,\, \dots \, , \, h_d\right) \sigma_h$, where 
    \[\text{$\sigma_h= \sigma_{p}$ \quad and \quad  
    $h_x= \begin{cases} g_{q_j}, \quad &\text{if $x=x_j$ for some cycle $\tau_j$ and $q_j\in \post(f)$}\\ \one, \quad &\text{otherwise}.\end{cases}$}\]
    Then the automorphisms $g_p$ and $h$ are conjugate in $\AT$.   
    \item\label{item: img-gen-iii} If the $p$-petals $[\gamma_p]$, $p\in P(f)$, generate $\pi_1(\MM, z_0)$, then the subgroup of $\Sym_d$ generated by $\{\sigma_p: p\in \post(f)\}$ acts transitively on $\sX$.

    \item\label{item: img-gen-iv} If the $p$-petals $[\gamma_p]$, $p\in P(f)$, form a standard generating set for $\pi_1(\MM,z_0)$, then there is an ordering of the generators $g_p$, $p\in \post(f)$, such that their product is the identity in~$\AT$.
\end{enumerate}
\end{prop}

\begin{rmk}
    We emphasize that Parts~\ref{item: img-gen-i} and \ref{item: img-gen-ii} specify the wreath recursion structure of the generator $g_p$ only up to conjugacy in $\AT$. However, when $f$ is a PCF polynomial that admits an \emph{invariant spider} (that is, a tree $\mathcal{S}\subset \rs$ with a unique branching point at $\infty$ and leaves at the finite postcritical points, satisfying $\mathcal{S}\subset f^{-1}(\mathcal{S})$ up to isotopy relative to $\post(f)$), then Part~\ref{item: img-gen-ii} can be strengthened to an equality of $g_p$ and $h$ for $p\in P(f)\setminus\{\infty\}$. For further details, we refer the reader to \cite[Sec.~6.8 and 6.9]{Nekrashevych2005}; see also \cite[Chap.~5]{Hlu_Thesis} for a discussion of the rational case.
\end{rmk}

\begin{proof}[Proof of Proposition~\ref{prop: img-gen-structure}]
    Part~\ref{item: img-gen-iii} is immediate from the fact that the permutations $\sigma_p$, $p\in \post(f)$, generate the monodromy group of the covering $f\colon \rs\setminus f^{-1}\big(\post(f)\big)\to \rs\setminus \post(f)$. 

    Part~\ref{item: img-gen-iv} is implied by the fact that, when $\{[\gamma_p]:\, p\in \post(f)\}$ is a standard generating set of $\pi_1(\MM,z_0)$, 
    one can order the loops $\gamma_p$ so that their product forms a null-homotopic loop in~$\MM$.

    Now suppose that $\gamma_p=\beta_p\alpha_p\beta_p^{-1}$, where $\alpha_p$ is a small (positively oriented) circular loop around $p$ based at $z_p$, and $\beta_p\subset \MM$ is a path connecting $z_0$ to $z_p$. Note that $f$ acts locally as the power map $z\mapsto z^{d_j}$ at each preimage $q_j\in f^{-1}(p)$ (after a conformal change of coordinates in both domain and target). Hence, $f^{-1}(\alpha_p)$ has exactly $s=|f^{-1}(p)|$ connected components given by ``small'' topological circles  $\widetilde\alpha_{p,1},\dots,\widetilde\alpha_{p,s}$ around $q_1,\dots, q_s$, respectively, so that  $f\colon \widetilde\alpha_{p,j}\to\alpha_p$ is a covering of degree $d_j$ for each $j=1,\dots,s$. This immediately implies Part~\ref{item: img-gen-i}, which describes the monodromy action of $[\gamma_p]$. 
    
    To prove Part~\ref{item: img-gen-ii}, suppose that $x_j\in \supp(\tau_j)$ for a cycle $\tau_j$ in the disjoint cycle decomposition of~$\sigma_p$. Let $\widetilde\beta_{p,j}$ denote the $f$-lift of the path $\beta_p$ starting at $\Lambda(x_j)\in f^{-1}(z_0)$, and let $\widetilde{z}_{p,j}\in \widetilde \alpha_{p,j}$ be its endpoint. Using Equation~\eqref{eq: img-restriction}, we deduce that 
    \[
    (x_jv)\cdot [\gamma_p]^{d_j}=x_j(v\cdot[\ell_{x_j}\widetilde\beta_{p,j}\widetilde{\alpha}_{p,j}\widetilde\beta_{p,j}^{-1}\ell_{x_j}^{-1}]) \quad \text{for all $v\in \sX^*$},
    \]
    where we view $\widetilde{\alpha}_{p,j}$ as a (positively oriented) circular loop based at $\widetilde{z}_{p,j}$. Observe that the concatenated path $\ell_{x_j}\widetilde\beta_{p,j}\widetilde{\alpha}_{p,j}\widetilde\beta_{p,j}^{-1}\ell_{x_j}^{-1}$ is either 
    \begin{itemize}
        \item a $q_j$-petal in $\MM$ when $q_j\in\post(f)$, or
        \item a null-homotopic loop in $\MM$ when $q_j\notin \post(f)$.
    \end{itemize}
    Hence, the cyclic section product $(g_p)_{\tau_j}=(g_p)^{d_j}|_{x_j}$ is a $W$-conjugate of $g_{q_j}$ in the former case, and equals $\one$ in the latter case. The statement of Part~\ref{item: img-gen-ii} now follows from Proposition~\ref{prop: conjugacy-in-W}. This completes the proof of Proposition~\ref{prop: img-gen-structure}. 
\end{proof}

The following statement is an immediate corollary of Part~\ref{item: img-gen-i} in the proposition above.

\begin{cor}\label{cor: IMG(poly) has odometer}
If $f$ is a polynomial and $[\gamma_\infty]$ is an $\infty$-petal, then the corresponding element $g_\infty\in \IMG(f)$ is an odometer, that is, $g_\infty$ acts transitively on each level of the tree $\sT_\sX$.
\end{cor}

\section{Model generators and groups}\label{Section:Model gps}

Let $f\in \C(z)$ be a PCF polynomial of degree $3$. By the discussion in Sections~\ref{subsec: img-discr} and \ref{subsec-generators}, throughout this and the subsequent sections, we identify the ternary rooted tree $T$, whose vertices are the solutions of the equations $f^n(z) - t=0$, with the tree of words $\sT_\sX$ in the alphabet $\sX=\{1,2,3\}$. Accordingly, we view the profinite geometric iterated monodromy group $G^{\mathrm{geom}}(f)$ as a (closed) subgroup of the automorphism group $\AT = \Aut(\sT_\sX)$. For notational convenience, we will simply write $T$ in place of $\sT_\sX$.

\subsection{\ref{main-assumption}-restricted model groups}

In this subsection, we introduce \emph{model subgroups} of $\AT$ that we will utilize to study the profinite geometric iterated monodromy groups of PCF cubic polynomials. The key idea is that the \emph{model generators}, i.e., the generators of these model groups, are prescribed by the combinatorial structure of the ramification portrait of $f$; see Definition~\ref{defn-ramification-portrait} and Remark~\ref{rem: can-model-group}. We begin by recording the following immediate corollary of Propositions~\ref{prop: conjugacy-in-W} and~\ref{prop: img-gen-structure}, which describes the wreath recursion structure of a standard generating set of $G^{geom}(f)$ up to conjugacy.

\begin{cor}\label{cor: cubic-model-generators}
    Let $f\in \C(z)$ be a PCF cubic polynomial with the critical set $C:=\crit(f)$ and postcritical set $P:=\post(f)$. Suppose $\{g_p: p\in P\}$ is a  standard generating set for the profinite geometric iterated monodromy group $G^{geom}(f)$ induced by $p$-petals in $\rs\setminus P$. Then any subset of $|P|-1$ standard generators topologically generates the group $G^{geom}(f)$. Moreover, each generator $g_p$ is a $\AT$-conjugate of (exactly) one of the following nine automorphisms: 
    
     \begin{tabular}{@{}lll}
    \caseitem{1a}  &   $(\,\,\;\;\one,\,\, \;\;\one, \,\;\; \one\,)(1\, 2\, 3)$ & \text{if $f^{-1}(p)=\{q_1\}$ for some $q_1\in C\setminus P$;}\\
    \caseitem{1b}  &    $(\,g_{q_1}, \,\, \;\;\one, \,\;\;\one \,)(1\, 2\, 3)$ & \text{if $f^{-1}(p)=\{q_1\}$ for some $q_1\in C\cap P$;} \\
    \caseitem{2a}   &  $(\,\,\;\;\one,\,\,\;
    \;\one, \,\;\; \one\,)(1\, 2)$ & \text{if $f^{-1}(p)=\{q_1,q_2\}$ for $q_1\in C\setminus P$ and $q_2\notin C\bigcup P$;} \\
    \caseitem{2b}  &  $(\,g_{q_1}, \,\, \;\;\one, \,\;\;\one\,)(1\, 2)$ & \text{if $f^{-1}(p)=\{q_1,q_2\}$ for $q_1\in C\cap P$ and $q_2\notin C\bigcup P$;} \\
    \caseitem{2c}  &     $(\,\,\;
     \;\one, \,\,\;\;\one, \, g_{q_2})(1\, 2)$ &\text{if $f^{-1}(p)=\{q_1,q_2\}$ for $q_1\in C\setminus P$ and $q_2\in P\setminus C $;} \\
    \caseitem{2d}  &    $(\,g_{q_1},\,\,\;\; \one,\,g_{q_2})(1\, 2)$ &\text{if $f^{-1}(p)=\{q_1,q_2\}$ for $q_1\in C\cap P$ and $q_2\in P\setminus C$;} \\
       \caseitem{3a} &   $(\,g_{q_1}, \,\,\;\; \one, \,\;\;\one \,)$ &\text{if $f^{-1}(p)=\{q_1,q_2,q_3\}$ for $q_1\in P\setminus C$ and $q_2,q_3\notin C\bigcup P$;}\\
       \caseitem{3b}  &   $(\,g_{q_1}, \,g_{q_2}, \,\;\;\one\,)$ & \text{if $f^{-1}(p)=\{q_1,q_2,q_3\}$ for $q_1,q_2\in P\setminus C$ and $q_3\notin C\bigcup P$;}\\
       \caseitem{3c}  &   $(\,g_{q_1},\, g_{q_2}, \,g_{q_3})$ & \text{if $f^{-1}(p)=\{q_1,q_2,q_3\}$ for $q_1,q_2,q_3\in P\setminus C$.}
    \end{tabular}
\end{cor}

\begin{rmk} In particular, since $f$ is a polynomial, $\infty \in C\cap P$ and the corresponding standard generator $g_\infty$ falls under Case~\ref{case:1b}. Moreover, if $f$ satisfies Assumption~\ref{main-assumption}, then for each finite postcritical point $p$, the associated standard generator $g_p$ falls into one of the following cases:
\begin{itemize}
    \item Case~\ref{case:2a} or ~\ref{case:2b}, when $p$ is the image of a critical point;
    \item Case~\ref{case:3a} or ~\ref{case:3b}, otherwise.  
\end{itemize}
\end{rmk}

We now introduce model subgroups of $\AT$ for PCF cubic polynomials $f$ satisfying Assumption~\ref{main-assumption}. 

\begin{defn}\label{defn-group}
    Let $r\geq 0$ be an integer. A \ref{main-assumption}\emph{-restricted} (\emph{cubic polynomial}) \emph{model group} is a profinite subgroup \[G=\llangle a,b, c_1,\dots,c_r\rrangle\] of $\AT=\Aut(T)$ topologically generated by $(r+2)$ recursively defined elements
    \begin{align}\label{eq: top-generators}\begin{matrix}
    a&=&(x,\,\one,\,\one)(1\,2),& \\
    b&=&(\one,\,\one,\,y)(2\,3), &\\
    c_i&=&(c_{i,1},\,c_{i,2},\,c_{i,3}), & 1 \leq i \leq r,
    \end{matrix}\end{align}
    that satisfy the following conditions:
    \begin{enumerate}[label= (Y\arabic*)]
    \item\label{cond:Y1} all sections $x,y,c_{i,j}$ are in $\{\one,a,b,c_1,\dots,c_r\}$;
    \item\label{cond:Y2} each generator $\ell \in \{a,b,c_1,\dots,c_r\}$ occurs exactly once among the sections $x,y,c_{i,j}$;
    \item\label{cond:Y3} for each $i=1,\dots, r$, at least one of the sections $c_{i,1}$, $c_{i,2}$, or $c_{i,3}$ is trivial;
    \item\label{cond:Y4} for each $i=1,\dots, r$, there exists some $v\in \sX^*$ such that $c_i|_v=a$ or $c_i|_v=b$.
    \end{enumerate}
    We call the elements $a,b, c_1,\dots,c_r$ as above the \emph{model generators} of $G$.
\end{defn}

\begin{rmk}\label{rmk-23}
    We note that the choice of the permutations $(1\, 2)$ and $(2\, 3)$ as the first-level actions of $a$ and $b$, respectively, is customary and ensures that $G$ acts transitively on $\sX$; compare Part~\ref{item: img-gen-iii} of Proposition~\ref{prop: img-gen-structure}. Furthermore, Condition~\ref{cond:Y4} guarantees that none of the generators $c_i$ is trivial.
\end{rmk}

The following statement now easily follows from Proposition~\ref{prop: promoting-recursion-conjugacy} and Corollary~\ref{cor: cubic-model-generators}. 

\begin{prop}\label{prop: model-group-from-map}
    Let $f\in\C(z)$ be a PCF cubic polynomial satisfying Assumption~{\normalfont \ref{main-assumption}} with $P:=\post(f)$, and let $\{g_p: p\in P\}$ be a standard generating set of $G^{geom}(f)$ induced by $p$-petals in $\rs\setminus P$. Then there is a {\normalfont \ref{main-assumption}}-restricted model group topologically generated by a set $\{ \widehat{g}_p: p\in P\setminus \{\infty\}\}\subset \AT$ such that $\widehat g_p \sim_\AT g_p$ for each $p\in P\setminus\{\infty\}$.
\end{prop}

\begin{proof}
    We describe first how to associate a \ref{main-assumption}-restricted model group with the given map~$f$. By Assumption~\ref{main-assumption}, $f$ has two finite critical points with distinct images, say $p_a$ and $p_b$.

Let us consider $|P|-1$ automorphisms $\widehat{g}_p\in \AT$ with $p\in P\setminus \{\infty\}$ recursively defined in the following way:
\begin{itemize}
    \item For the critical value $p_a$, we have $0 \leq |f^{-1}(p_a)\cap P|\leq 1$, and we set
     \[
    \widehat{g}_{p_a}= \begin{cases}
    (\,\,\;\;\one, \;\; \one, \;\; \one ) (1\, 2), \quad\text{if $f^{-1}(p_a)\cap P = \emptyset$}\\
    (\,\widehat{g}_{q_1},\;\; \one, \;\; \one ) (1\, 2)
    , \quad \text{if $f^{-1}(p_a)\cap P = \{q_1\}.$}
    \end{cases}
    \]   
    \item Similarly, for the critical value $p_b$, we have $0 \leq |f^{-1}(p_b)\cap P|\leq 1$, and we set 
    \[
    \widehat{g}_{p_b}= \begin{cases}
    (\,\;\;\one, \;\; \one, \;\; \,\one ) (2\, 3), \quad\text{if $f^{-1}(p_b)\cap P = \emptyset$}\\
    (\,\;\;\one,\;\; \one, \,\widehat{g}_{q_1} ) (2\, 3)
    , \quad \text{if $f^{-1}(p_b)\cap P = \{q_1\}$.}
    \end{cases}
    \]   
    \item Finally, for each $p\in P\setminus \{p_a,p_b,\infty\}$, we have $1\leq |f^{-1}(p)\cap P|\leq 2$, and we set 
    \[
    \widehat{g}_p= \begin{cases}
    (\,\widehat{g}_{q_1},\,\;\; \one, \;\; \one ), \quad\quad \text{if $f^{-1}(p)\cap P = \{q_1\}$}\\
    (\,\widehat{g}_{q_1},\, \widehat{g}_{q_2}, \,\;\one)
    , \quad\quad \text{if $f^{-1}(p)\cap P = \{q_1,q_2\}$.}
    \end{cases}
    \]
\end{itemize}

It is straightforward to verify that the constructed automorphisms $\widehat{g}_p$ satisfy Conditions~\ref{cond:Y1}--\ref{cond:Y4} (with $a=g_{p_a}$ and $b=g_{p_b}$). In particular, each automorphism $\widehat{g}_p$ occurs as a section of $\widehat{g}_{f(p)}$ and of no other element from the set $S^{model}(f):=\{\widehat{g}_p: p\in P\setminus \{\infty\}\}$. Hence, the subgroup $G^{model}(f)\subset \AT$ topologically generated by $S^{model}(f)$ is a \ref{main-assumption}-restricted model group. Furthermore, it follows directly from Proposition~\ref{prop: promoting-recursion-conjugacy} and Corollary~\ref{cor: cubic-model-generators} that each $\widehat{g}_p$ is a $\AT$-conjugate of the standard generator $g_p$ of $G^{geom}(f)$ induced by a $p$-petal in $\rs\setminus P$. This finishes the proof of Proposition~\ref{prop: model-group-from-map}.
\end{proof}

\begin{rmk}\label{rem: can-model-group}
We emphasize that the proof of Proposition~\ref{prop: model-group-from-map} assigns a \emph{canonical} \linebreak \ref{main-assumption}-restricted model group $G^{model}(f)$ to a given PCF cubic polynomial $f$ (satisfying Assumption~\ref{main-assumption}), up to relabeling the postcritical points of $f$. In fact, $G^{model}(f)$ is determined by the ramification portrait of $f$. Conversely, for each \ref{main-assumption}-restricted model group $G=\llangle a,b,c_1,\dots,c_r\rrangle$ as in Definition~\ref{defn-group}, one can find a PCF cubic polynomial $f\in \C(z)$ satisfying Assumption~\ref{main-assumption} and a bijection $$\varphi\colon \{g_p: p\in P(f)\setminus \{\infty\}\} \to \{a,b,c_1,\dots,c_r\}$$ between the standard generators of $G^{geom}(f)$ and the model generators of $G$ such that \[g_p \sim_\AT \varphi(g_p) \quad  \text{for each $p\in P(f)\setminus \{\infty\}$}.\] This follows readily from the work of Floyd et al. \cite{Floyd2022}, which shows that any \emph{abstract polynomial portrait} is realized as a ramification portrait of some PCF polynomial. 
\end{rmk}

\subsection{Basic properties of model groups}
We collect here some basic properties of \ref{main-assumption}-restricted model groups that will be needed for our study of invariable generation in Section~\ref{sec: inv-generation}, particularly in the proofs of Theorems~\ref{thm-conjugates} and~\ref{thm-simultaneous-conjugation}.

\begin{lemma}\label{lemma-model-groups}
    Let $G=\llangle a, b, c_1,\dots, c_r\rrangle \subset \AT$ be a {\normalfont \ref{main-assumption}}-restricted model group as in Definition~\ref{defn-group}, and let $\Gamma := \langle a,b, c_1,\ldots,c_r\rangle$ be the corresponding countable dense subgroup. Then the following is true:
    \begin{enumerate}[label=(\roman*),font=\normalfont]
    \item\label{item: model-i} $\Gamma$ and $G$ are self-similar subgroups of $\AT$.
    \item\label{item: model-ii} $\Gamma$ and $G$ contain an odometer. In fact, the product of all their (topological) generators $a,b, c_1,\ldots,c_r$ in any order is an odometer. 
    \item\label{item: model-iii}
    We have $[a,b]=a^{-1}b^{-1}ab=(\one, \, \one, \one) (1\,2\, 3) \in \Gamma\subset G$.
    \end{enumerate}
\end{lemma}

\begin{proof}
    Part~\ref{item: model-i} is an immediate consequence of Lemma~\ref{EXDOC: lem: G is self-similar}. Part~\ref{item: model-ii} follows from the fact that the model generators $a,b,c_1,\dots,c_r$ specify a \emph{kneading automaton} (see \cite[Def.~6.7.4]{Nekrashevych2005} for the definition), and thus their product in any order is an odometer by \cite[Cor.~6.7.7]{Nekrashevych2005}. Finally, Part~\ref{item: model-iii} follows by a direct computation. 
\end{proof}

The following lemma, along with Corollaries~\ref{cor: special-element} and \ref{cor: conjugacy-in-G}, provides key technical ingredients for the proof of Theorem~\ref{thm-conjugates} in Section~\ref{sets-of-conjugates}.

\begin{lemma} \label{EXDOC: Le: forms of elements of G}
    Let $H\subseteq G$ be arbitrary subgroups of $\AT$ such that $H_1= \Sym_3$ and $G$ is self-similar. If for all $g\in G$ there exists some $h\in H$ of the form $h=(g,*,\one)$, then for all $g\in G$ the group $H$ also contains elements of the form $$(g,\one,*),\, (*,\one,g),\, (*,g,\one),\, (\one,*,g), \, \textrm{and} \;\; (\one,g,*),$$
    where each $*$ represents some (possibly different) element of $G$.
    
    Moreover, the analogous claim still holds if we replace $g\in G$ and $h \in H$ with $g\in G_{n-1}$ and $h \in H_n$, respectively, for any $n \geq 1$.
\end{lemma}

\begin{proof}
    The proof is straightforward, and we outline it only for the first claim. Suppose that for all $g\in G$ we have some element of the form $(g,*,\one)$ inside $H$. Let us fix an arbitrary $g\in G$. As $H_1=\Sym_3$ and the group $G\supseteq H$ is self-similar, we can find some $h=(h_1,h_2,h_3)(13)\in H$ with each section $h_i\in G$. Then $h_3^{-1}gh_3\in G$ and, by our assumption, there exists some $\kappa\in G$ such that $(h_3gh_3^{-1},\kappa,\one)\in H$. Using Lemma~\ref{lem: conj-action}, we then have $$h(h_3^{-1}gh_3,\kappa,\one)h^{-1}=(\one,h_2\kappa h_2^{-1},g) \in H.$$ 
    So, we have shown that for all $g\in G$ the group $H$ contains an element of the form $(\one,*,g)$. The proofs for $(*,\one,g)$, $(*,g,\one)$, $(g,\one,*)$, and $(\one,g,*)$ follow in a similar manner.
\end{proof}

The next lemma states that \ref{main-assumption}-restricted model groups contain elements of special form; compare the hypotheses of Lemma~\ref{EXDOC: Le: forms of elements of G}. This will be of crucial importance for the proofs of our main results.

\begin{lemma}\label{lem: model-self-replicativity}
     Let $G=\llangle a,b,c_1,\dots,c_r\rrangle$ be a {\normalfont \ref{main-assumption}}-restricted model group as in Definition~\ref{defn-group}. Then
     for all $g\in G$ the group $G$ contains elements of the form 
     $$(g,*,\one),\, (g,\one,*),\, (*,\one,g),\, (*,g,\one),\, (\one,*,g), \, \textrm{and} \;\; (\one,g,*),$$
    where each $*$ represents some (possibly different) element of $G$. In particular, the group $G$ is self-replicating. 
    
    Furthermore, the analogous statements hold for the corresponding countable dense subgroup $\Gamma:=\langle a,b,c_1,\dots,c_r\rangle$.
\end{lemma}

\begin{proof}
    First, we claim that, for each model generator $\ell\in \{a,b,c_1,\dots, c_r\}$, the group $\Gamma\subset G$ contains some element of the form $(\ell, *, \one)$. Indeed, by Condition~\ref{cond:Y2}, we know that $\ell$ appears within the set $\{x,y\} \cup \{ c_{i,j}: \,i=1,\dots, r,\, j=1,2,3\}$ of sections of the generators. Our claim then follows by considering different cases:
    \begin{itemize}
        \item If $\ell=x$, we consider the element $a^2=(x,x,\one)$. 
        \item If $\ell=y$, we may consider  the conjugate $(1\, 2\, 3) \, b^2\,  (1\, 2\, 3)^{-1} = (y, y, 1)$, since $(1\, 2\, 3)\in G$ by Part~\ref{item: model-iii} of Lemma~\ref{lemma-model-groups}. 
        \item If $\ell=c_{i,j}$, then, by Condition~\ref{cond:Y3}, the generator $c_i$ itself or its conjugate by $(1\, 2\, 3)^{\pm1}$ must be of the form $(\ell, *, \one)$ or $(\ell, \one , *)$. In the latter case, we may conjugate one more time by $b$ to get an element of the desired form $(\ell, *, \one)$, and the claim follows. 
    \end{itemize}

    We conclude that for all $g\in \Gamma$, the countable group $\Gamma$ contains an element of the form $(g, *, \one)$; in particular, $\Gamma$ is self-replicating (see Definition~\ref{defn-self-similar-replicating}). The analogous statement for the closure $G=\overline{\Gamma}$ now follows by a compactness argument. Namely, given an arbitrary $g\in G$, let us consider a sequence $\{g_m\} \subset \Gamma$ converging to $g$. We know that for each $g_m$ there exists an element $(g_m,\kappa_m, \one) \in \Gamma\subset G$ for some $\kappa_m\in \Gamma$. Since $G$ is compact, the sequence $\{(g_m,\kappa_m, \one)\}$ has an accumulation point in $G$, which then must be of the form  $(g,\kappa, \one)$. Since $G$ is self-similar by Part~\ref{item: model-i} of Lemma~\ref{lemma-model-groups}, we have $\kappa\in G$, which completes the argument. 

    Finally, since $\Gamma_1=G_1=\Sym_3$, the existence of elements of the other desired forms for $\Gamma$ and $G$ follows immediately from Lemma~\ref{EXDOC: Le: forms of elements of G}, with $H=\Gamma$ and $H=G$, respectively. 
    \end{proof}

 \begin{cor}\label{cor: (g,ginv,1)}
     Let $G=\llangle a,b,c_1,\dots,c_r\rrangle$ be a {\normalfont \ref{main-assumption}}-restricted model group as in Definition~\ref{defn-group}. Then, for all $g\in G$, the element $(g,g^{-1},\one)$ and all its $\Sym_3$-conjugates belong to the commutator subgroup $[G,G]$. 

    Furthermore, the analogous statement holds for the corresponding countable dense subgroup $\Gamma:=\langle a,b,c_1,\dots,c_r\rangle$.
\end{cor}
\begin{proof}
     By Lemma~\ref{lem: model-self-replicativity}, $G$ contains the element $\kappa:=(g, \one, h)$ for some $h\in G$. A direct computation then shows that
    \[[\kappa,a^{-1}]=\kappa^{-1}a\kappa a^{-1}=(g^{-1},g,\one)\in [G,G]\]
    and
    \[[\kappa^{-1},a^{-1}]=\kappa a\kappa^{-1} a^{-1}=(g,g^{-1},\one)\in [G,G].\]
    Since $(1\, 2\, 3)\in [G,G]$ by Lemma~\ref{lemma-model-groups}, \ref{item: model-iii}, the desired statement for $G$ follows by conjugating with $(1\, 2\, 3)^{\pm1}$. The proof for the countable subgroup $\Gamma$ is completely analogous. 
\end{proof}

Finally, we show that the class of \ref{main-assumption}-restricted model groups is not invariant under conjugation in $W$. In particular, this implies that the profinite geometric iterated monodromy group $G^{geom}(f)$ in Theorem \ref{main-theorem} is conjugate to a \ref{main-assumption}-restricted model group, but need not be a \ref{main-assumption}-restricted model group itself.

\begin{cor}\label{cor-class-not-closed}
The class of {\normalfont \ref{main-assumption}}-restricted model groups is not closed under conjugation. Namely, for each group $G$ in this class there exists some $w\in W$ such that $wGw^{-1}$ is not a {\normalfont \ref{main-assumption}}-restricted model group. 
\end{cor}

\begin{proof}
    Let $G$ be an arbitrary {\normalfont \ref{main-assumption}}-restricted model group, and fix some $w\in \AT\setminus G$. Note that such an element $w$ exists, because $G$ is finitely topologically generated, whereas the whole automorphism $\AT$ is not. We claim that the group $H:=(w,\one,\one)G(w,\one,\one)^{-1}$ does not satisfy Definition~\ref{defn-group}. For the sake of contradiction, let us assume that it does.

    First, let us recall that $(1\,2\,3)\in G$ by Part~\ref{item: model-iii} of Lemma~\ref{lemma-model-groups}. Hence, we get $$(w,\one,\one)(123)(w^{-1},\one,\one)=(w,\one,w^{-1})(123)\in H.$$ By the assumption and Part~\ref{item: model-i} of Lemma~\ref{lemma-model-groups}, we conclude that $w\in H$. Using again the assumption and Lemma~\ref{lem: model-self-replicativity}, we know that there exists some $h\in H$ such that $(w,h,\one)\in H$. Conjugating this element back to $G$, we obtain $$(w^{-1},\one,\one)(w,h,\one)(w,\one,\one)=(w,h,\one)\in G.$$ Hence, by the self-similarity of $G$, we have $w\in G$, which is a contradiction.
\end{proof}

\section{Invariable generation and conjugacy of groups}\label{sec: inv-generation}

Our goal in this section is to prove Theorem~\ref{main-theorem}, which states that profinite geometric iterated monodromy groups associated to cubic PCF polynomials satisfying Assumption~\ref{main-assumption} are finitely invariably generated and, up to conjugacy, they are classified by the isomorphism class of their ramification portrait. This result follows from two statements for \ref{main-assumption}-restricted model groups introduced in Definition~\ref{defn-group}: first, in Section~\ref{sets-of-conjugates}, we prove that a \ref{main-assumption}-restricted model group is finitely invariably generated; second, in Section~\ref{sec-conjugate-subgroups}, we prove that a subgroup $H$ of $\AT= \Aut(T)$, generated by $\AT$-conjugates of the model generators of a \ref{main-assumption}-restricted model group $G$, such that $H$ contains an odometer, is conjugate to $G$ by an element of $\AT$. The proof of Theorem~\ref{main-theorem} then follows in Section~\ref{sec-proof-main-theorem}.

\subsection{Finite invariable generating sets for \ref{main-assumption}-restricted model groups}\label{sets-of-conjugates}

Recall that $\sX = \{1,2,3\}$ and $T_n$ denotes a finite subtree of the ternary rooted tree $T = \sT_\sX$ composed of all vertices of the first $n$ levels. Moreover, for a group $G \subset \AT=\Aut(T)$ and $n\geq 0$, we set  $G_n:=\{g|_{T_n}: g\in G\}$. Throughout this section, when we say that an element $g\in W_n$ \emph{acts transitively on the subtree $T_n$} (or that $g \in W_n$ is a \emph{transitive element}), we mean that $g$ acts transitively on the $n$-th (and hence on any $k$-th, for $0 \leq k \leq n$) level set $\sX^n$ of $T_n$.

\begin{thm}\label{thm-conjugates}
    Let $G=\llangle a,b,c_1,\dots,c_r\rrangle \subset \AT$ be a \emph{\ref{main-assumption}}-restricted model group as in Definition~\ref{defn-group}. Then the following statements hold: 
    \begin{enumerate}[label=(\roman*),font=\normalfont]
    \item \label{inclusion-with-odometer} Suppose $A,B,C_1,\dots,C_r\in G$ are $G$-conjugates of $a,b,c_1,\dots, c_r$, respectively, and let $D \in G$ be an odometer. Then the group $H := \llangle A,B,C_1,\dots,C_r, D \rrangle$ equals $G$.
    \item \label{invariable-generating-set} The set $S = \{a,b,c_1,\ldots,c_r,d\}$, where $d\in G$ is any odometer, is an invariable generating set for $G$. In particular, we can take $d=abc_1\cdots c_r$.
    \end{enumerate}  
 \end{thm}

   \begin{rmk}
{\rm
Recall that, by Part~\ref{item: model-ii} of Lemma~\ref{lemma-model-groups}, the product $d$ of the model generators $a,b,c_1,\dots,c_r$ of the \ref{main-assumption}-restricted model group $G$ (in any order) is an odometer.
We note that an arbitrary odometer $D\in G$ need not be a conjugate of $d$ in $G$. 
} 
\end{rmk}

 \begin{cor}   \label{EXDOC: Thm: invariable generation} 
  Let $G=\llangle a,b,c_1,\dots,c_r\rrangle \subset \AT$ be a \emph{\ref{main-assumption}}-restricted model group as in Definition~\ref{defn-group}.
  Suppose $A\sim_G a, B\sim_G b$, and $C_i\sim_G c_i$ for all $1 \leq i \leq r$.  Then the group $H:=\llangle A,B,C_1,\dots,C_r\rrangle$ contains an odometer if and only if $H=G$. 
\end{cor}

The role of the odometer in the statements of Theorem \ref{thm-conjugates} is illustrated by the following simple example.

\begin{ex} \label{ex-counterexample}
    Let $G=\llangle a,b,c_1,...,c_r\rrangle$ be a {\normalfont \ref{main-assumption}}-restricted model group as in Definition~\ref{defn-group}. Upon replacing $b$ by $$B:=(ba)b(ba)^{-1}=(x,yx^{-1},\one)(12),$$ the group $H:=\llangle a,B,c_1,...,c_r\rrangle$ is not equal to $G$. Indeed, since each $c_i$ acts trivially on the first level, and $a|_{T_1}=(12)$, we get $H_1=\langle (12)\rangle\neq \Sym_3=G_1$. Hence $H\ne G$, and therefore $G$ is not invariably generated by $\{a,b,c_1,...,c_r\}$.
\end{ex}

The main steps of the proof of Theorem \ref{thm-conjugates} are as follows: 
\begin{enumerate}
    \item \label{step1-52} As $A,B,C_1,\dots,C_r$ are conjugate to $a,b,c_1,\dots,c_r$ in $G$, and $D\in G$, we have $H=\llangle A,B,C_1,\dots,C_r,D\rrangle \subseteq G$, so we have to show that $G\subseteq H$. The latter will follow from the the statement $G_n = H_n$ for all $n\geq 0$, which we will prove by induction.
    
    \item \label{step2-52} The group $G$ is self-similar, and, since $H$ contains an odometer $D$, we have $G_1 = H_1= {\mathcal S}_3$. The main technical step in the proof is to show that $H_n$ and $G_{n-1}$ satisfy the hypothesis of Lemma~\ref{EXDOC: Le: forms of elements of G}; that is, for each $g \in G_{n-1}$ the group $H_n$ contains an element of the form $(g,*,\one)$. To establish this, we repeatedly use Corollary~\ref{cor: special-element}.
    \item \label{step3-52} Once we have Step~\ref{step2-52}, we conjugate the generators $A|_{T_n},B|_{T_n},C_1|_{T_n},\dots,C_r|_{T_n}$ of $H_n$ by elements of $H_n$ constructed using Lemma~\ref{EXDOC: Le: forms of elements of G}, and show that these conjugates yield the generators $a|_{T_n},b|_{T_n},c_1|_{T_n},\ldots,c_r|_{T_n}$ of $G_n$. This proves Part~\ref{inclusion-with-odometer}.
    \item \label{step4-52} Part~\ref{invariable-generating-set} follows from the observation that any conjugate of the odometer $d$ in $G$ is again an odometer, so we may take such a conjugate as $D$ in Part~\ref{inclusion-with-odometer}.
\end{enumerate}

\begin{proof}[Proof of Theorem \ref{thm-conjugates}] 
 Let $G=\llangle a,b,c_1,\ldots,c_r \rrangle$ be a \ref{main-assumption}-restricted model group  as in Definition~\ref{defn-group}.

{\bf Step~\ref{step1-52}.}    
    We use induction to prove the following statement for all $n\geq0$:

\begin{enumerate}
[label=\normalfont{$(*_n)$}]
    \item\label{induction: invariable generation} For any $A,B,C_1,\dots,C_r\in G_n$ with $A\sim_{G_n} a|_{T_n},B\sim_{G_n} b|_{T_n},C_i\sim_{G_n} c_i|_{T_n}$, $1 \leq i \leq r$, and any $D\in G_n$ acting transitively on $T_n$, we have $\langle A,B,C_1,\dots,C_r,D\rangle =G_n$.
\end{enumerate}

    As $(*_0)$ trivially holds, we suppose that $(*_{n-1})$ is true for some $n \geq 1$. Take $A,B,C_1,\dots,C_r\in G_n$, where $A\sim_{G_n} a|_{T_n},B\sim_{G_n} b|_{T_n}$, and $C_i\sim_{G_n} c_i|_{T_n}$ for all $i$, and let $D\in G_n$ be any element that acts transitively on $T_n$. Define $H_n:=\langle A,B,C_1,\dots,C_r,D\rangle\subseteq G_n$. We will prove that $a|_{T_n},b|_{T_n}\in H_n$ and $c_i|_{T_n}\in H_n$ for all $1\leq i\leq r$.

    {\bf Step \ref{step2-52}.} First, we show that for every model generator $\ell\in \{a,b, c_1,\dots,c_r\}$ of $G$ there is an element $(\widehat{\ell}, *, \one) \in H_n$ with $\widehat{\ell}$ conjugate to $\ell|_{T_{n-1}}$ in $G_{n-1}$.
    
    Consider the elements $A^2,B^2,C_i\in H_n$, which satisfy the following relations:
    \[\begin{matrix*}[l]
    A^2\sim_{G_n} a^2|_{T_n}=(x|_{T_{n-1}},x|_{T_{n-1}},\one), &\\
    B^2\sim_{G_n} b^2|_{T_n}=(\one,y|_{T_{n-1}},y|_{T_{n-1}}), \;\text{and} &\\
    C_i\sim_{G_n} c_i|_{T_n}=(c_{i,1}|_{T_{n-1}},c_{i,2}|_{T_{n-1}},c_{i,3}|_{T_{n-1}}), & 1 \leq i \leq r.\\
    \end{matrix*}.\]

    In particular, they all act trivially on $T_1$. Note that $H_n|_{T_1}:=\{h|_{T_1}: h\in H_n\}=\Sym_3$, because $A=a|_{T_n}$ and $D$ act as a 2-cycle and a 3-cycle on $T_1$, respectively, and any 2-cycle and 3-cycle generate~$\Sym_3$. Hence, we can obtain $H_n$-conjugates $\widehat{A^2}=(\alpha,\alpha',\one)$ and $\widehat{B^2}=(\beta,\beta',\one)$ of $A^2$ and $B^2$, respectively, where $\alpha\sim_{G_{n-1}} x|_{T_{n-1}}$ and $\beta\sim_{G_{n-1}} y|_{T_{n-1}}$. 
     Indeed, suppose $A^2 = (\alpha_1,\alpha_2,\alpha_3)$ for some $\alpha_1, \alpha_2,\alpha_3\in G_{n-1}$.  Since $G$ is self-similar by Part~\ref{item: model-i} of Lemma~\ref{lemma-model-groups}, Corollary~\ref{cor: conjugacy-in-G} implies that there are elements $g',g''\in G_{n-1}$ and $\mu \in \Sym_3$ such that
     $$\alpha_{1\cdot\mu} = g' (x|_{T_{n-1}}) (g')^{-1}, \quad \alpha_{2\cdot\mu} = g'' (x|_{T_{n-1}}) (g'')^{-1}, \quad \text{and} \quad \alpha_{3\cdot\mu} = \one.$$    
    Then, by Corollary \ref{cor: special-element}, there exists $h = (h_1,h_2,h_3)\mu \in H_n$ such that
      $$\widehat{A^2}:= hA^2h^{-1} = \left(h_1g'(x|_{T_{n-1}}) (g')^{-1}h_1^{-1}, h_2g''(x|_{T_{n-1}}) (g'')^{-1}h_2^{-1}, \one \right) \in H_n.$$
     Since $G$ is self-similar, we have $h_1, g' \in G_{n-1}$, and so $\alpha := h_1g'(x|_{T_{n-1}})(g')^{-1}h_1^{-1} \sim_{G_{n-1}} x|_{T_{n-1}}$.
     
     The argument is similar for $\widehat{B^2}$. Again, by a similar argument, we obtain $\widehat{C_i}=(\gamma_{i,1},\gamma_{i,2},\one) \in H_n$, $1 \leq i \leq r$, where every $\gamma_{i,k}\sim_{G_{n-1}} c_{i,j(i,k)}|_{T_{n-1}}$ for some $j(i,k)\in\{1,2,3\}$. We also create conjugates $\widehat{C_i'}=(\delta_{i,2},\delta_{i,1},\one)\in H_n$ of $C_i$, $1 \leq i \leq r$, with $\delta_{i,k}\sim_{G_{n-1}} c_{i,j(i,k)}|_{T_{n-1}}$ for each $k=1,2$, by conjugating each $\widehat{C_i}$ with an element $\kappa\in H_n$ such that $\kappa|_{T_1}=(1\, 2)$. Conditions~\ref{cond:Y1} and~\ref{cond:Y2} now imply that, 
     for every model generator $\ell$ of $G$, there exists a unique element $L\in\{ \widehat{A^2},\widehat{B^2},\widehat{C_i},\widehat{C_i'}: 1\leq i \leq r\}\subset H_n$
     such that $L=(\widehat{\ell}, *, \one)$ with $\widehat{\ell}\sim_{G_{n-1}} \ell|_{T_{n-1}}$.

    Now consider the group 
      $$\widehat H_n:=\langle \widehat{A^2},\widehat{B^2},\widehat{C_i},\widehat{C_i'},D^3: 1 \leq i \leq r\rangle\subset H_n,$$ and write $D^3=(d_1,d_2,d_3)$, where each $d_i$ must act transitively on $T_{n-1}$, since $D$ acts transitively on $T_n$. By the form of the generators of $\widehat H_n$, the induction hypothesis $(*_{n-1})$ implies that $\langle \alpha,\beta,d_1,\gamma_{i,1},\delta_{i,2} : 1 \leq i \leq r\rangle=G_{n-1}$, and thus every element $g\in G_{n-1}$ appears as the section at $1$ of some product of the generators of $\widehat{H}_n$. Since each such generator acts trivially on $T_1$ and $G_n \supset H_n \supset  \widehat H_n$ is self-similar, we have shown the following claim. 
    
    \emph{Claim.} For any $g\in G_{n-1}$, there exist $g',g''\in G_{n-1}$ such that $(g,g',g'')\in H_n$.

    We now show that for any $g \in G_{n-1}$ there is an element $(g,*,\one) \in H_n$, i.e., the hypothesis of Lemma~\ref{EXDOC: Le: forms of elements of G} is satisfied for $G_{n-1}$ and $H_n$.

    Indeed, we have some $p,q,u_1,\dots,u_r,z_1,\dots,z_r\in G_{n-1}$ such that $p\alpha p^{-1}=x|_{T_{n-1}}$, $q\beta q^{-1}=y|_{T_{n-1}}$, and $u_i\gamma_{i,1}u_i^{-1}=c_{i,j(i,1)}|_{T_{n-1}}$, $z_i\delta_{i,2}z_i^{-1}=c_{i,j(i,2)}|_{T_{n-1}}$ for $1 \leq i \leq r$.

    By the claim above, we can choose $p',p'',q',q''u_i',u_i'',z_i',z_i''\in G_{n-1}$ such that $P:=(p,p',p'')$, $Q:= (q,q',q''), U_i:=(u_i,u_i',u_i''), Z_i:=(z_i,z_i',z_i'')\in H_n$, $1 \leq i \leq r$. Now, by conjugation, we obtain elements in $H_n$ of the following forms:
    \begin{align*}P\widehat{A^2}P^{-1}&=(x|_{T_{n-1}},*,\one),\\
    Q\widehat{B^2}Q^{-1}&=(y|_{T_{n-1}},*,\one),\\
    U_i\widehat{C_i}U_i^{-1}&=(c_{i,j(i,1)}|_{T_{n-1}},*,\one),   \; 1 \leq i \leq r, \\
    Z_i\widehat{C_i'}Z_i^{-1}&=(c_{i,j(i,2)}|_{T_{n-1}},*,\one),  \; 1 \leq i \leq r.\end{align*}

    But then, by Condition~\ref{cond:Y2}, for each model generator $\ell$ of $G$,  we have found an element of the form $(\ell|_{T_{n-1}}, *, \one)\in H_n$.  Therefore, by multiplying these elements, we can obtain $(g,*,\one) \in H_n$ for all $g\in G_{n-1}$. Then Lemma \ref{EXDOC: Le: forms of elements of G} applies, and we also have elements of the forms $(*,g,\one),(g,\one,*),(*,\one,g),(\one,g,*)$ and $(\one,*,g)$ in $H_n$.

   {\bf Step \ref{step3-52}.} We now show that the restrictions of $ a,b,c_1,\ldots,c_r $ to $T_n$ are in $H_n$.
    
    To obtain $c_i|_{T_n}\in H_n$ for a fixed $i$, we first conjugate $C_i$ by some $h_i\in H_n$ such that $(h_iC_ih_i^{-1})|_k \sim_{G_{n-1}} c_{i,k}|_{T_{n-1}}$ for each $k\in\{1,2,3\}$, which is possible by  Corollary~\ref{cor: special-element}. By replacing $C_i$ with this conjugate, we have that $C_i = (\widehat{c}_{i,1},\widehat{c}_{i,2},\widehat{c}_{i,3}) $ with $\widehat{c}_{i,k} \sim_{G_{n-1}}c_{i,k}$, $k = 1,2,3$. Note that $\widehat{c}_{i,3}$ need not be a trivial section, and this is the reason we do not use $\widehat{C}_i$, obtained previously, but construct a new set of conjugates. However, by Condition~\ref{cond:Y3} and Corollary~\ref{cor: conjugacy-in-G}, for at least one $k\in\{1,2,3\}$ we must have $\widehat{c}_{i,k}=\one$. For simplicity, we assume that $\widehat{c}_{i,3}=c_{i,3}=\one$. (If this is not the case, the argument remains valid upon replacing 3 by either 1 or 2.) 
    Choose $u_{i,1},u_{i,2}\in G_{n-1}$ such that $u_{i,k}\gamma_{i,k}u_{i,k}^{-1}=c_{i,k}|_{T_{n-1}}$ for $k=1,2$. Since Lemma~\ref{EXDOC: Le: forms of elements of G} applies, we can find $(u_{i,1},\one,*),(\one,u_{i,2},*)\in H_n$. We then derive that
    $$(u_{i,1},\one,*)(\one,u_{i,2},*)C_i(\one,u_{i,2},*)^{-1}(u_{i,1},\one,*)^{-1}=c_i|_{T_n}\in H_n.$$

    We now show that $b|_{T_n}\in H_n$. Then $a|_{T_n}\in H_n$ follows in a similar fashion. Note that $B=g(b|_{T_n})g^{-1}$ for some $g=(g_1,g_2,g_3)\sigma\in G_n$. If $\sigma$ is non-trivial, we replace $B$ with $hg(b|_{T_n})g^{-1}h^{-1}$ for some $h\in H_n$ with $h|_{T_1}=\sigma^{-1}$, which is possible by Corollary~\ref{cor: special-element} since $H_n|_{T_1} = \Sym_3$. Now we may assume that $B=g(b|_{T_n})g^{-1}$ with $g=(g_1,g_2,g_3)$ for some $g_1,g_2,g_3\in G_{n-1}$. We write out:
    $$B=g(b|_{T_n})g^{-1}=(\one,g_2g_3^{-1},g_3yg_2^{-1})(2\,3).$$ 
    By Lemma~\ref{EXDOC: Le: forms of elements of G}, we can find $(*,g_2,\one),(*,\one,g_3)\in H_n$, and thus 
    $$(*,\one,g_3)^{-1}(*,g_2,\one)^{-1}B(*,g_2,\one)(*,\one,g_3)=(\one,\one,y|_{T_{n-1}})(2\,3)=b|_{T_n}\in H_n.$$

    We conclude that the generators $a|_{T_n},b|_{T_n},c_i|_{T_n}$, $1 \leq i \leq r$, of $G_n$ are all contained in $H_n$. Hence $G_n\subseteq H_n$, and thus we have $G_n=H_n$, finishing the induction.

    \medskip

    Suppose now that $A,B,C_1,\dots,C_r\in G$ are $G$-conjugates of $a,b,c_1,\dots,c_r$, respectively, and let $D\in G$ be an odometer. Set $H:=\llangle A,B,C_1,\dots, C_r, D\rrangle$. Since $D|_{T_n}$ acts transitively on $T_n$ and \ref{induction: invariable generation} is true for all $n\geq 0$, we conclude that $H_n=\langle A|_{T_n},B|_{T_n}, C_1|_{T_n},\dots, C_r|_{T_n}, D|_{T_n}\rangle=G_n$ for all $n\geq 0$. Lemma~\ref{lem-same-closures} then implies that $H=G$, as desired.

    {\bf Step \ref{step4-52}.} To obtain Part~\ref{invariable-generating-set}, let us suppose that $A,B,C_1,\dots,C_r$ are $G$-conjugates of the model generators $a,b,c_1,\dots,c_r$, respectively, and that $D$ is a $G$-conjugate of an odometer $d\in G$. But then $D$ is an odometer itself, and thus $\llangle A,B,C_1,\ldots,C_r,D\rrangle = G$ by Part~ \ref{inclusion-with-odometer}. It follows that $\{a,b,c_1,\ldots,c_r,d\}$ is an invariable generating set for $G$, which completes the proof of the theorem.  
\end{proof}

\subsection{Pairwise conjugacy of generators implies conjugacy of groups}\label{sec-conjugate-subgroups} 
We will now prove that, if a profinite group $H \subset W$ has a set of topological generators that are pairwise conjugate to the model generators of a  \ref{main-assumption}-restricted model group $G$ (see Definition~\ref{defn-group}), and $H$ contains an odometer, then $H$ is conjugate to $G$ in $W$. 

\begin{thm} \label{thm-simultaneous-conjugation}
Let $G=\llangle a,b,c_1,\ldots,c_r \rrangle \subset W$ be a \emph{\ref{main-assumption}}-restricted model group  as in Definition~\ref{defn-group}. Suppose $A,B,C_1,\dots,C_r\in W$ are $W$-conjugates of $a,b,c_1,\dots,c_r$, respectively, such that the group $$H := \llangle A,B,C_1,\dots,C_r\rrangle$$ contains an odometer $D$. Then there exists some $w\in W$ such that $$w^{-1}Aw,w^{-1}Bw,w^{-1}C_1w,\dots,w^{-1}C_rw$$ are $G$-conjugates of $a,b,c_1,\dots,c_r$, respectively, and $w^{-1} \, H \, w=G$.
\end{thm}

The strategy of the proof is as follows:
\begin{enumerate}
    \item \label{step1-55} The main technical step is to prove the following statement: For each $n \geq 0$, we can find an element $w_n\in W_n$ that simultaneously conjugates each generator \linebreak $L\in \{A|_{T_n},B|_{T_n},C_1|_{T_n},...,C_r|_{T_n}\}$ of $H_n$ to a $G_n$-conjugate of the corresponding generator $\ell\in \{a|_{T_n},b|_{T_n},c_1|_{T_n},...,c_r|_{T_n}\}$ of $G_n$. The conjugating elements $w_n\in W_n$ and $g_\ell\in G_n$, such that $w_n^{-1}Lw_n=g_\ell\ell g_\ell^{-1}$, are constructed by induction on the level $n$. 
    
    \item \label{step2-55} First, we conjugate the given generators of $H_n$ into a more convenient form: specifically, we want each first-level section of each generator $L$ of $H_n$ to be a conjugate of the respective section of the corresponding generator $\ell$ of $G_n$. The main tools used in this step are  Corollary~\ref{cor: special-element} and Proposition~\ref{prop: conjugacy-in-W}. We also verify that it is sufficient to prove the induction step for these newly  constructed elements. 
    
    \item \label{step3-55} Next, we implement the induction step by using the special form of our new generators of $H_n$. To construct the desired conjugating elements $w_n\in W_n$ and $g_\ell\in G_n$, we invoke the induction hypothesis, which provides us with the sections of $w_n$ and $g_\ell$'s, and  Lemma~\ref{lem: model-self-replicativity}, which ensures that $G_n$ contains elements with such sections. 

    \item \label{step4-55} Finally, having constructed a sequence $\{w_n\}$ of  conjugating elements as in Step~\ref{step1-55}, we use compactness to construct an element $w\in W$ that simultaneously conjugates each generator of $H$ to a $G$-conjugate of the corresponding generator of $G$. The desired equality $w^{-1}Hw=G$ then follows from Corollary~\ref{EXDOC: Thm: invariable generation}.
\end{enumerate}

\begin{proof}
Let $G=\llangle  a,b,c_1,\ldots,c_r\rrangle$ be a \ref{main-assumption}-restricted model group as in Definition~\ref{defn-group}.

{\bf Step \ref{step1-55}.}
    We will show the following statement for all $n\geq 0$:

   \begin{enumerate}[label=\normalfont{$(*_n)$}]
    \item\label{induction: simultaneus conjugation}
        For any $A,B,C_1,\dots,C_r\in W_n$ conjugate to $a|_{T_n},b|_{T_n},c_1|_{T_n},\dots,c_r|_{T_n}$, respectively, in $W_n$, 
        such that $\langle A,B,C_1,\dots,C_r\rangle$ contains an element $D$ acting transitively on $T_n$, 
        there exist $w_n\in W_n$ and $X,Y,Z_1,\dots,Z_r \in G_n$ such that 
          \begin{align*}
              A &=(w_nX)(a|_{T_n})(w_nX)^{-1}, & \\  B& =(w_nY)(b|_{T_n})(w_nY)^{-1}, \, \textrm{and} \\ C_i& =(w_nZ_i)(c_i|_{T_n})(w_nZ_i)^{-1}, 1 \leq i \leq r. &  
          \end{align*} 
    \end{enumerate}

    As $(*_0)$ trivially holds, we assume that $(*_{n-1})$ is true for some $n\geq 1$. Let us now consider arbitrary $A,B,C_1,\dots,C_r\in W_n$ with $A\sim_{W_n} a|_{T_n}, B\sim_{W_n} b|_{T_n}, C_1\sim_{W_n} c_1|_{T_n}, \dots, C_r\sim_{W_n} c_r|_{T_n}$, and suppose that $H_n:=\langle A,B,C_1,\dots,C_r\rangle$ contains an element acting transitively on $T_n$.

    {\bf Step \ref{step2-55}.}
    First, note that \ref{induction: simultaneus conjugation} is true for the given $A,B,C_1,\ldots, C_r$ if and only if it is true for their conjugates 
    by the same element $g\in W_n$, by choosing $gw_n$ instead of $w_n$. Moreover, since $H_n$ contains a transitive element $D$, and all $C_i$, $1 \leq i \leq r$, act trivially on $T_1$, we must have $A|_{T_1}\neq B|_{T_1}$, that is, $A|_{T_1}$ and $B|_{T_1}$ represent distinct transpositions.

    Choose $\rho\in {\mathcal S}_3$ such that $\rho (A|_{T_1})\rho^{-1}=(1\,2)$ and $\rho(B|_{T_1})\rho^{-1}=(2\,3)$. By replacing the given $A,B,C_i$, $1 \leq i \leq r$, with their conjugates by $\rho$, we may assume that $A|_{T_1}=(12)$ and $B|_{T_1}=(23)$. By Corollary~\ref{cor: special-element}, we then have $A=(u_1,u_2,\one)(1\,2)$ and $B=(\one,v_2,v_1)(2\,3)$ for some $u_i,v_j\in W_{n-1}$. Now we conjugate every $A,B,C_1,\ldots,C_r$ by $g:=(u_2,\one,v_2)\in W_n$. For $A$ and $B$, this will give: $$gAg^{-1}=(u_2,\one,v_2)(u_1,u_2,\one)(12)(u_2^{-1},\one,v_2^{-1})=(u_2u_1,\one,\one)(12)$$ and $$gBg^{-1}=(u_2,\one,v_2)(\one,v_2,v_1)(23)(u_2^{-1},\one,v_2^{-1})=(\one,\one,v_2v_1)(23).$$
    To avoid cumbersome notation, we again replace the given elements $A,B,C_1,\ldots,C_r$ with their conjugates by $g$. By Proposition~\ref{prop: conjugacy-in-W}, we can now write $A=(\alpha,\one,\one)(12)$ and $B=(\one,\one,\beta)(23)$, where $\alpha\sim_{W_{n-1}} x|_{T_{n-1}}$ and $\beta\sim_{W_{n-1}} y|_{T_{n-1}}$.

    Next, we conjugate $C_1,\ldots,C_r$ further to arrange their sections in a convenient order. First, note that the first-level sections of each $C_i$ are pairwise conjugate to the first-level sections of $c_i|_{T_n}$ in some order by Proposition~\ref{prop: conjugacy-in-W}. Then, since $\langle A|_{T_1},B|_{T_1}\rangle=\Sym_3$, we can use Corollary~\ref{cor: special-element} to find some $g_i\in \langle A,B\rangle$ such that the conjugate $C_i':=g_iC_ig_i^{-1}=(\gamma_{i,1},\gamma_{i,2},\gamma_{i,3})$ satisfies $\gamma_{i,j}\sim_{W_{n-1}} c_{i,j}|_{T_{n-1}}$ for each $j\in \{1,2,3\}$. Moreover, by construction of the $C_i$'s, we still have $H_n=\langle A,B,C_1',\dots,C_r'\rangle$. 

     We now show that if \ref{induction: simultaneus conjugation} holds for the elements $A,B,C_1',\dots,C_r'$, then it also holds for the given $A,B,C_1,\dots,C_r$. Indeed, suppose we have proven $(*_n)$ for $A,B,C_1',\dots,C_r'$. Then there exist $w_n'\in W_n$, and elements $X',Y',Z_1',\dots,Z_r'\in G_n$ such that 
     \begin{equation}\label{eq: from C_primes to C}
      \begin{aligned}
          A & =(w_n'X')(a|_{T_n})(w_n'X')^{-1}, &\\
          B & =(w_n'Y')(b|_{T_n})(w_n'Y')^{-1}, \, \textrm{and} & \\
          C_i' & =(w_n'Z_i')(c_i|_{T_n})(w_n'Z_i')^{-1}, \, 1 \leq i \leq r .&
      \end{aligned}
      \end{equation}
    
  The first two equations of \eqref{eq: from C_primes to C} imply that $\langle A ,B\rangle \subset w_n'G_n {w_n'}^{-1}$, and thus for each $1 \leq i \leq r$ we can choose $h_i\in G_n$ such that $g_i=w_n'h_i(w_n')^{-1}$. Now let $w_n:=w_n', X:=X',Y:=Y'$, and $Z_i:=h_i^{-1}Z_i'$, $1 \leq i \leq r$. Then $A=(w_nX)(a|_{T_n})(w_nX)^{-1}$ and $B=(w_nY)(b|_{T_n})(w_nY)^{-1}$. At the same time, for each $C_i$ we have: 
  \begin{align*}
            C_i &= g_i^{-1}C_i'g_i\\
            &= w'_nh_i^{-1}(w'_n)^{-1}(w'_nZ_i')(c_i|_{T_n})(w'_nZ_i')^{-1}w'_nh_i(w'_n)^{-1}\\
            &= w'_nh_i^{-1}Z_i'(c_i|_{T_n})Z_i'^{-1}h_i(w'_n)^{-1}\\
            &= (w_nZ_i)(c_i|_{T_n})(w_nZ_i)^{-1}.
        \end{align*} 
        It follows that \ref{induction: simultaneus conjugation} is true for $A,B,C_1,\dots,C_r$ as well. So we can replace the given $C_i$ by $C_i'$ for each $1\leq i \leq r$. Observe also that the group $H_n=\langle A,B,C_1,\dots,C_r\rangle$ does not change and still contains an element acting transitively on $T_n$, because the conjugate of a transitive element is still transitive.

        {\bf Step \ref{step3-55}.} 
        By the discussion above, we may assume that $A=(\alpha, \one, \one) (1\, 2)$,  $B=(\one, \one, \beta) (2\, 3)$, and $C_i=(\gamma_{i, 1}, \gamma_{i,2}, \gamma_{i,3})$, $1\leq i\leq r$, where $\alpha\sim_{W_{n-1}} x|_{T_{n-1}}$, $\beta\sim_{W_{n-1}} y|_{T_{n-1}}$, and  $\gamma_{i,j}\sim_{W_{n-1}} c_{i,j}|_{T_{n-1}}$ for all $1\leq i \leq r,j\in \{1,2,3\}$. Moreover, the group $H_n=\langle A,B,C_1,\dots,C_r\rangle$ contains an element $D$ acting transitively on $T_n$.

        Set $\delta:=D^3|_1$. By the assumption on the form of generators in the previous paragraph,  
        $\delta\in \langle \alpha, \beta, \gamma_{i,j}\rangle$ and it acts transitively on $T_{n-1}$. Furthermore,  Conditions~\ref{cond:Y1} and \ref{cond:Y2} imply that, besides the trivial element, the set $\{\alpha, \beta, \gamma_{i,j}\}$ contains a conjugate of each of the generators $a|_{T_{n-1}}, b|_{T_{n-1}}, c_1|_{T_{n-1}}, \dots,  c_r|_{T_{n-1}}$ of $G_{n-1}$, and only those. It follows that we may apply the induction hypothesis to the corresponding group $\langle \alpha, \beta, \gamma_{i,j}\rangle$. Hence, we must have $w_{n-1}\in W_{n-1}$ and $X',Y',Z_{i,j}'\in G_{n-1}$ such that 
         \begin{align*}
             \alpha & =(w_{n-1}X')(x|_{T_{n-1}})(w_{n-1}X')^{-1}, & \\ 
             \beta & =(w_{n-1}Y')(y|_{T_{n-1}})(w_{n-1}Y')^{-1},\, \textrm{and} & \\
             \gamma_{i,j} &=(w_{n-1}Z_{i,j}')(c_{i,j}|_{T_{n-1}})(w_{n-1}Z_{i,j}')^{-1}, \, 1\leq i \leq r, \, j \in \{1,2,3\}. &
         \end{align*}
        Here, if a section $x$, $y$, or $c_{i,j}$ is trivial, we have the freedom  to choose the corresponding element from $X',Y',Z_{i,j}'$ in an arbitrary way.

    Using Lemma~\ref{lem: model-self-replicativity}, we can now find $X=(X',X',*)=(X',\one,*)(\one,X',*)\in G_n$ and, similarly, $Y=(*,Y',Y')\in G_n$. By Condition~\ref{cond:Y3}, for each $1\leq i\leq r$, the section $c_{i,j}$ is trivial for some $j\in\{1,2,3\}$, and thus the corresponding $Z'_{i,j}\in G_{n-1}$ can be chosen in such a way that $Z_i:=(Z'_{i,1},Z'_{i,2},Z'_{i,3})\in G_n$ using the same lemma. Finally, let $w_n:=(w_{n-1},w_{n-1},w_{n-1})\in W_n$. We then have: 
     \begin{align*}
         (w_nX)(a|_{T_n})(w_nX)^{-1} &=\big((w_{n-1}X')(x|_{T_{n-1}})(w_{n-1}X')^{-1},\one,\one\big)(1\,2)=(\alpha,\one,\one)(1\,2)=A, \\
         (w_nY)(b|_{T_n})(w_nY)^{-1} & =\big(\one,\one,(w_{n-1}Y')(y|_{T_{n-1}})(w_{n-1}Y')^{-1}\big)(2\,3)=(\one,\one,\beta)(2\,3)=B,
     \end{align*}
     and
     \begin{align*}
         (w_nZ_i)(c_i|_{T_n})(w_nZ_i)^{-1} &=(w_{n-1}Z'_{i,1},w_{n-1}Z'_{i,2},w_{n-1}Z'_{i,3})(c_{i,1}|_{T_{n-1}},c_{i,2}|_{T_{n-1}},c_{i,3}|_{T_{n-1}}) \\ &(w_{n-1}Z'_{i,1},w_{n-1}Z'_{i,2},w_{n-1}Z'_{i,3})^{-1}  =(\gamma_{i,1},\gamma_{i,2},\gamma_{i,3})=C_i,
     \end{align*}
    for all $1\leq i \leq r$. This finishes the induction step, proving $(*_n)$ for all $n\geq0$.

    {\bf Step \ref{step4-55}.} 
    Suppose now that $A,B,C_1,\dots,C_r\in W$ are $W$-conjugates of $a,b,c_1,\dots,c_r$, respectively, such that the group $H := \llangle A,B,C_1,\dots,C_r\rrangle$ contains an odometer $D$. For each $n\geq 0$ and  $\ell_n\in \{a|_{T_n},b|_{T_n},c_1|_{T_n},...,c_r|_{T_n}\}$, let $w_n\in W_n$ and $g_{\ell_n}\in G_n$ 
    be the elements such that the conjugate $(w_ng_{\ell_n})(\ell|_{T_n})(w_ng_{\ell_n})^{-1}$ is equal to the corresponding element $L_n\in\{A|_{T_n},B|_{T_n},C_1|_{T_n},...,C_r|_{T_n}\}$. The existence of such conjugating elements is guaranteed by applying \ref{induction: simultaneus conjugation} to the restrictions $A|_{T_n},B|_{T_n},C_1|_{T_n},...,C_r|_{T_n}$.

    For each $n\geq 0$ and  $\ell\in \{a, b, c_1, \dots, c_r\}$, let $\widetilde w_n\in W$ and $\widetilde g_{\ell_n}\in G$ be some elements such that $\widetilde w_n|_{T_n} = w_n$ and $\widetilde g_{\ell_n}|_{T_n} = g_{\ell_n}$.
    By compactness of $W$ and $G$, we can find convergent subsequences of $\{\widetilde w_n\}$ and $\{\widetilde g_{\ell_n}\}$'s, all indexed by the same subsequence. Let $w\in W$ and $g_\ell\in G$ denote the respective limits. By the convergence, for each model generator $\ell\in \{a, b, c_1, \dots, c_r\}$, we have that the conjugate
    $(wg_{\ell})\ell(wg_{\ell})^{-1}$ is equal to the corresponding element $L\in\{A,  B, C_1, \dots, C_r\}$. It follows that the group $w^{-1}Hw$ is generated by $G$-conjugates of the model generators of $G$. Since $w^{-1}Hw$ contains the odometer $w^{-1}Dw$, we conclude that $w^{-1}Hw=G$ by Corollary~\ref{EXDOC: Thm: invariable generation}. This completes the proof of Theorem~\ref{thm-simultaneous-conjugation}.
\end{proof}

\subsection{Proof of Theorem~\ref{main-theorem}}\label{sec-proof-main-theorem}

Let $K$ be a number field, and let $f\in K(z)$ be a PCF cubic polynomial satisfying Assumption~\ref{main-assumption}.

\emph{Part \ref{item-MT-1}}. Suppose $S:=\{g_p: p\in P(f)\}$ is a standard generating set of the profinite geometric iterated monodromy group $G^{geom}(f)$. Then $G^{geom}(f)=\llangle g_p: p \in P(f) \setminus \{\infty\}\rrangle$ by Corollary~\ref{cor: cubic-model-generators}, and, by (the proof of) Proposition~\ref{prop: model-group-from-map}, the standard generators $g_p$, $p \in P(f) \setminus \{\infty\}$, are $W$-conjugates of the model generators of a \ref{main-assumption}-restricted model group. More precisely, there exist a \ref{main-assumption}-restricted model group $G=G^{model}(f)=\llangle a, b, c_1, \dots, c_r\rrangle$ as in Definition~\ref{defn-group} and a bijection 
   $\varphi\colon \{g_p : p \in P(f) \setminus \{\infty\}\} \to \{a,b,c_1,\dots, c_r\}$ such that $g_p\sim_{W} \varphi(g_p)$ for each $p\in P(f) \setminus \{\infty\}$. In particular, $r=|P(f)|-3$, and the group $G$ is determined by the ramification portrait of $f$. Since the element $g_\infty \in \llangle g_p : p \in P(f) \setminus \{\infty\}\rrangle = G^{geom}(f)$, corresponding to the $\infty$-petal, is an odometer by Corollary~\ref{cor: IMG(poly) has odometer}, Theorem~\ref{thm-simultaneous-conjugation} implies that there exists $w \in W$ such that $w^{-1}G^{geom}(f)w=G$, and, furthermore, that each $w^{-1}g_pw$, $p\in P(f)\setminus \{\infty\}$, is a $G$-conjugate of the corresponding model generator $\varphi(g_p)$ of $G$. Since $w^{-1}g_\infty w\in G$ is an odometer, Part~\ref{invariable-generating-set} of Theorem~\ref{thm-conjugates} implies that the set $w^{-1}Sw=\{w^{-1}g_pw: p\in P(f)\}$ is an invariable generating set for $G$. It then readily follows that $S=\{g_p: p\in P(f)\}$ is an invariable generating set for $G^{geom}(f)=wGw^{-1}$.

\emph{Part~\ref{item-MT-2}}. 
 Let $f'$ be another cubic polynomial over a number field such that its ramification portrait is isomorphic to the ramification portrait of $f$. In particular, $f'$ is PCF and satisfies Assumption~\ref{main-assumption}.
 By the proof of Part~\ref{item-MT-1}, there are \ref{main-assumption}-restricted model groups $G, G'$ and elements $w,w' \in W$ such that $w^{-1}G^{geom}(f)w=G$ and $(w')^{-1}G^{geom}(f'){w'}=G'$. As the ramification portraits of $f$ and $f'$ are isomorphic, we can assume that $G=G'$; see Remark~\ref{rem: can-model-group}. Therefore, $G^{geom}(f)$ and $G^{geom}(f')$ are conjugate by $w'w^{-1} \in W$. This completes the proof of Theorem~\ref{main-theorem}. 

\section{Branch and torsion properties}\label{branch-torsion}

We now prove Theorem~\ref{thm-reg-branch} and Corollary~\ref{polynomial-branch}, which concern the branch and torsion properties in \ref{main-assumption}-restricted model groups and in profinite geometric iterated monodromy groups associated to cubic PCF polynomials satisfying Assumption~\ref{main-assumption}.

\subsection{Branch properties} We refer the reader to Definition~\ref{weakly-branch} for the notions of regular branch and regular weakly branch subgroups of the automorphism group $\AT$. Below, we show that \ref{main-assumption}-restricted model groups are regular branch over the closure of their commutator subgroups. The proof consists of the following steps:  First, we show a weaker statement that if $G=\llangle a,b,c_1,\ldots,c_r\rrangle$ satisfies Definition~\ref{defn-group}, then it is regular weakly branch over $\mathfrak{C}:=\overline{[G,G]}$; see Proposition~\ref{Prop: G is branch over [G,G]}. As a by-product, we obtain that the corresponding countable dense subgroup $\Gamma := \langle a,b,c_1,\ldots,c_r\rangle$ is regular weakly branch over its commutator subgroup $[\Gamma,\Gamma]$; see Corollary~\ref{cor-weakly-branch-countable}. Afterward, we establish that $\mathfrak{C}$ has finite index in $G$, and thus $G$ is regular branch over $\mathfrak{C}$, in Theorem~\ref{Thm: M_r is always regular branch}.

\begin{prop}\label{Prop: G is branch over [G,G]}
Let $G$ be a \emph{\ref{main-assumption}}-restricted model group, and
    let $\mathfrak C:=\overline{[G,G]}$ denote the closure of its commutator subgroup. Then $G$ is regular weakly branch over $\mathfrak C$. 
    Moreover, $\mathfrak C$ contains the infinite iterated wreath product $$[\Alt_3]^\infty:=\cdots \wr \Alt_3\wr \Alt_3\wr \Alt_3$$ embedded into $\AT=[\Sym_3]^\infty$ via the natural embedding $\Alt_3\hookrightarrow  \Sym_3$,  where $\Alt_3$ denotes the alternating subgroup of $\Sym_3$.
\end{prop}

\begin{proof}
    Suppose $G = \llangle a,b,c_1,\ldots,c_r\rrangle$, where $r\geq 0$ and  $a,b,c_1,\ldots,c_r$ are as in Definition~\ref{defn-group}. Then $[a,b] = a^{-1}b^{-1}ab=(1\,2\,3) \in \mathfrak C$ by Part~\ref{item: model-iii} of Lemma~\ref{lemma-model-groups}. 

    Let $g,h\in G$ be arbitrary. By Lemma~\ref{lem: model-self-replicativity}, there exist $p,q\in G$ such that $(g,p,\one),(h,\one,q)\in G$. Then
    \begin{align}\label{commutator-sections}
        [(g,p,\one),(h,\one,q)]=(g^{-1}h^{-1}gh,p^{-1}p,q^{-1}q)=([g,h],\one,\one).
    \end{align} 
    It follows that $[G,G]\times \{\one\}\times \{\one\}\subset \mathfrak C$. By taking the closure of $[G,G]$, we obtain that $\mathfrak C\times \{\one\}\times \{\one\}\subset \mathfrak C$. Finally, by conjugation with $(123)\in \mathfrak C$, we get $\mathfrak C\times \mathfrak C\times \mathfrak C\subset \mathfrak C$, and thus $G$ is regular weakly branch over $\mathfrak{C}$. 
    
    As $\Alt_3=\langle (123)\rangle \subset \mathfrak C$, we conclude that the group $\mathfrak C$ contains the $n$-fold wreath product $[\Alt_3]^n := \Alt_3 \wr \cdots \wr \Alt_3$ for each $n\geq 1$. The groups $[\Alt_3]^n$ form an increasing sequence of subgroups of $\mathfrak C$, and the closure of $\bigcup_{n\geq1} [\Alt_3]^n$ is a compact subgroup of $\mathfrak C$, equal to the infinite iterated wreath product 
    $[\Alt_3]^\infty$ of alternating groups. This completes the proof of the proposition. 
\end{proof}

\begin{cor}\label{cor-weakly-branch-countable}
     Let $\Gamma:=\langle a,b,c_1,\ldots,c_r\rangle$ be a countable dense subgroup of a \emph{\ref{main-assumption}}-restricted model group $G=\llangle a,b,c_1,\ldots,c_r\rrangle$ as in Definition \ref{defn-group}. Then $\Gamma$ is weakly regular branch over its commutator subgroup $[\Gamma,\Gamma]$.
\end{cor}

\begin{proof}
    The proof is the same as in Proposition~\ref{Prop: G is branch over [G,G]}, without taking the closures. Indeed, $[a,b] = (1\, 2\, 3) \in [\Gamma,\Gamma]$ by  Part~\ref{item: model-iii} of Lemma~\ref{lemma-model-groups}, and Lemma~\ref{lem: model-self-replicativity} holds for $\Gamma$ as well. Hence, the corollary follows by the same argument.
\end{proof}

The next result holds only for (profinite) \ref{main-assumption}-restricted model groups, and will be used to show that these groups are regular branch, strengthening Proposition~\ref{Prop: G is branch over [G,G]}.

\begin{prop}\label{Prop: We can multiply the generators to have order 2}
     Let $G=\llangle a,b,c_1,\ldots,c_r\rrangle$ be a \emph{\ref{main-assumption}}-restricted model group as in Definition~\ref{defn-group}, and let $\mathfrak C:=\overline{[G,G]}$. Then, for each generator $\ell\in \{a,b,c_1,\dots,c_r\}$ of $G$, there exist $p,q\in \mathfrak C$ such that the product $p\ell q$ has order 2. 
\end{prop}

\begin{proof}
    As in the proof of Proposition~\ref{Prop: G is branch over [G,G]}, we denote by $[\Alt_3]^n$ the $n$-fold wreath product $\Alt_3\wr \Alt_3\wr \cdots \wr \Alt_3$ for each $n\geq 0$, with $[\Alt_3]^0:=\{\one\}$. Recall also that $[\Alt_3]^n\subset [\Alt_3]^{n+1}=[\Alt_3]^n\wr \Alt_3\subset [\Alt_3]^\infty\subset \mathfrak{C}$ for each $n\geq 0$. 
    
    First, we use induction to prove the following assertion for all $n\geq 0$:

    \begin{enumerate}[label=\normalfont{$(*_n)$}]
    \item\label{induction: order-2-products}
        For each $\ell_n\in \{a|_{T_n},b|_{T_n},c_1|_{T_n},\dots,c_r|_{T_n}\}$, there exist $p_{\ell_n},q_{\ell_n}\in [\Alt_3]^n\subset \mathfrak{C}$ such that $(p_{\ell_n}\,\ell_n\, q_{\ell_n})^2=\one$. 
    \end{enumerate}

    As this statement is clearly true when $n=0$, we suppose it is proven up to some $n\geq 0$. Since the generators $a,b,c_1,\dots,c_r$ satisfy Condition~\ref{cond:Y1}, for each $\ell_{n}\in \{x|_{T_{n}},y|_{T_{n}},c_{i,j}|_{T_{n}}\}$ we can choose $p_{\ell_{n}},q_{\ell_{n}}\in [\Alt_3]^{n}$ so that $(p_{\ell_{n}}\, \ell_{n} \, q_{\ell_{n}})^2=\one$, based on the induction hypothesis. (Here, if $\ell_n$ is trivial, we simply set $p_{\ell_n}=q_{\ell_n}=\one$.) 
    
    Now consider an arbitrary $\ell_{n+1}\in \{a|_{T_{n+1}},b|_{T_{n+1}},c_1|_{T_{n+1}},\dots,c_r|_{T_{n+1}}\}$. For $\ell_{n+1}=a|_{T_{n+1}}$, we set $p_{a|_{T_{n+1}}}:=(p_{x|_{T_{n}}},\one,\one)\in [\Alt_3]^{n+1}$ and $q_{a|_{T_{n+1}}}:=(\one,q_{x|_{T_{n}}},\one)(1\, 3\, 2) \in [\Alt_3]^{n+1}$, so that: 
    $$p_{a|_{T_{n+1}}}(a|_{T_{n+1}})\,q_{a|_{T_{n+1}}}=(p_{x|_{T_{n}}},\one,\one)\,(x|_{T_{n}},\one,\one)(1\, 2)\,(\one,q_{x|_{T_{n}}},\one)(1\,3\,2)$$ $$=(p_{x|_{T_{n}}}(x|_{T_{n}})q_{x|_{T_{n}}},\one,\one)(2\,3).$$ 
    By the definition of $p_{x|_{T_n}}$ and $q_{x|_{T_n}}$, the element above has order $2$. Mimicking this argument for $\ell_{n+1}=b|_{T_{n+1}}$, we can also find corresponding $p_{b|_{T_{n+1}}}, q_{b|_{T_{n+1}}}\in [\Alt_3]^{n+1}$. For $\ell_{n+1}=c_i|_{T_{n+1}}$, we set 
    $$p_{c_i|_{T_{n+1}}}:=(p_{c_{i,1}|_{T_{n}}},p_{c_{i,2}|_{T_{n}}},p_{c_{i,3}|_{T_{n}}})\in [\Alt_3]^{n+1}$$ and $$q_{c_i|_{T_{n+1}}}:=(q_{c_{i,1}|_{T_{n}}},q_{c_{i,2}|_{T_{n}}},q_{c_{i,3}|_{T_{n}}})\in [\Alt_3]^{n+1}.$$ Then, it can be readily seen that $(p_{c_i|_{T_{n+1}}}(c_i|_{T_{n+1}})\,q_{c_i|_{T_{n+1}}})^2=\one$, finishing the proof of the induction step.
     
    Fix now an arbitrary $\ell\in \{a,b,c_1,\dots,c_r\}$, and consider the sequences $\{p_n:=p_{\ell|_{T_{n}}}\}$ and $\{q_n:=q_{\ell|_{T_{n}}}\}$ with $(p_n\,\ell|_{T_n}\,q_n)^2=\one$ provided by \ref{induction: order-2-products}.  Since the group $[\Alt_3]^\infty$ is closed, it is compact, and thus the sequences $\{p_{n}\}, \{q_{n}\} \subset [\Alt_3]^\infty$ have convergent subsequences with the same index sets and limits $p, q\in [\Alt_3]^\infty\subset \mathfrak C$, respectively. Then the convergence ensures that $(p \ell q)^2 = \one$. Finally, note that Condition~\ref{cond:Y4} implies that $\ell$ acts as a $2$-cycle on the set $\{v1, v2, v3\}$ for some $v\in \sX^*$. It then follows that $\ell\notin [\Alt_3]^\infty$, and thus $p \ell q \neq \one$. We conclude that $p\ell q$ has order $2$, which finishes the proof.
    \end{proof}

Finally, we prove that \ref{main-assumption}-restricted model groups are regular branch. 

\begin{thm}\label{Thm: M_r is always regular branch}
Let $G=\llangle a,b,c_1,\ldots,c_r\rrangle$ be a \emph{\ref{main-assumption}}-restricted model group as in Definition~\ref{defn-group}. Then the closure $\mathfrak C := \overline{[G,G]}$ of its commutator subgroup has finite index in $G$, and, consequently, $G$ is regular branch over $\mathfrak C$. Moreover, the index of $\mathfrak C$ in $G$ is at most $2^{r+2}$, where $r+2$ is the number of topological generators of $G$. 
\end{thm}

\begin{proof}
   By Proposition~\ref{Prop: We can multiply the generators to have order 2},
   for each $\ell\in \{a,b,c_1,\dots,c_r\}$, we can find $p,q\in \mathfrak{C}$ such that the product $p\ell q$ has order $2$. Since $\mathfrak C = \overline{[G,G]}$, it then follows that the coset $\ell^2\mathfrak{C}$ equals $\mathfrak C$. Now we claim that the quotient group $G/\mathfrak C$ is generated by the cosets $a\mathfrak C,b\mathfrak C,c_1\mathfrak C,\dots,c_r\mathfrak C$. Indeed, if $g\in \Gamma:=\langle a,b,c_1,\dots,c_r\rangle \subset G$, then we can write $g\mathfrak C=a^{k_1}b^{k_2}c_1^{k_3}\cdots c_r^{k_{r+2}}\mathfrak C$ for some $k_i\in \{0,1\}$. Therefore, the elements of the countable subgroup $\Gamma$ are contained in finitely many cosets of $\mathfrak C$. Since $\Gamma$ is dense in $G$, any $g\in G$ is the limit of a sequence $\{g_n\}$ in~$\Gamma$. As $\mathfrak C\subset G$ is closed, we know that $G/\mathfrak C$ is a Hausdorff space, so in particular limits are well-defined. By the continuity of the quotient map, the coset $g\mathfrak C$ must be the limit of $\{g_n\mathfrak C\}$. But then, as $g_n\mathfrak C$ has only finitely many options for each $n$, the sequence $\{g_n\mathfrak C\}$ must eventually become constant. It then follows that $g\mathfrak C$ equals one of the cosets $a^{k_1}b^{k_2}c_1^{k_3}\cdots c_r^{k_{r+2}}\mathfrak C$ with $k_i\in\{0,1\}$. Hence, $\mathfrak C$ has finite index in $G$, bounded above by $2^{r+2}$. Proposition~\ref{Prop: G is branch over [G,G]} then implies that $G$ is regular branch over $\mathfrak{C}$, which completes the proof of the theorem.
\end{proof}

\subsection{Torsion}

In this subsection, we study torsion in \ref{main-assumption}-restricted model groups. First, let us recall that the automorphism group $\AT=\Aut(T)$ of the ternary rooted tree $T=\sT_\sX$, where $\sX=\{1,2,3\}$, is identified with the infinite iterated wreath product $[\Sym_3]^\infty$. Consequently, if $g\in \AT$ is a torsion element, its order must divide some power of $6$, and hence it will be equal to $2^m3^n$ for some $m,n\geq 0$. The next result shows that all such orders are realizable in every \ref{main-assumption}-restricted model group.

\begin{thm}\label{thm-torsion}
 Let $G=\llangle a,b,c_1,\ldots,c_r\rrangle$ be a \emph{\ref{main-assumption}}-restricted model group as in Definition~\ref{defn-group}, let $\Gamma=\langle a,b,c_1,\ldots,c_r\rangle$ be the corresponding countable dense subgroup, and let $\mathfrak C:=\overline{[G,G]}$. Then the following holds:
    \begin{enumerate}[label=(\roman*),font=\normalfont]
        \item \label{item: torsion-i} For each $n \geq 0$, the commutator subgroup $[\Gamma,\Gamma] \subset \Gamma$ contains an element of order~$3^n$. 
        \item  \label{item: torsion-ii} For each $m,n \geq 0$, the group $\mathfrak{C}\subset G$ contains an element of order $2^m3^n$, and thus all torsion orders realizable in $\Aut(T)$ are represented in the profinite group $G$.
    \end{enumerate}
\end{thm}

\begin{proof}
In the proof of Proposition~\ref{Prop: G is branch over [G,G]}, we showed that for each $n \geq 0$, the group $[\Gamma,\Gamma]$ contains the $n$-fold wreath product $[\Alt_3]^n = \Alt_3 \wr \cdots \wr \Alt_3$. In particular, for each fixed $n\geq 0$, $[\Gamma,\Gamma]$ contains the element $h$ such that $h|_{T_n} = g|_{T_n}$, where $g=(g,\one,\one)(1\,2\,3)$, and such that for all $v \in \sX^{n}$ the sections $h|_v$ are trivial. Since $g$ acts transitively on $T_n$, we conclude that $h$ has order~$3^n$, which establishes Part~\ref{item: torsion-i} of the theorem. Also, since $[\Alt_3]^n \subset \mathfrak{C}=\overline{[G,G]}$, it follows that $\mathfrak{C}\subset G$ contains elements of order $3^n$ for all $n \geq 0$.

To prove Part~\ref{item: torsion-ii}, we first show the following claim.

\smallskip

\emph{Claim}. There is an element $a'=(\one, \one, x') (1\, 2)\in G$ with $(x')^2=\one$. 

Indeed, we can multiply the generator $a=(x,\one,\one)(1\,2)$ by $(1\, 2\, 3)\in \mathfrak{C}$ on both sides to get the element $(\one,\one,x)(1\,2)\in G$. Since $x\in \{\one, a,b,c_1,\dots,c_r\}$ by Condition~\ref{cond:Y1}, Proposition~\ref{Prop: We can multiply the generators to have order 2} implies that there exists $p,q\in \mathfrak{C}$ such that $(pxq)^2=\one$. (Here, if $x$ is trivial, we simply set $p=q=\one$.) Since $(\one,\one,p), (\one,\one,q)\in \mathfrak{C}\subset G$ by Proposition~\ref{Prop: G is branch over [G,G]}, we conclude that $$a':=(\one ,
\one,p)(\one,\one,x)(1\,2)(
\one,\one, q)=(\one,\one, pxq) (1\, 2) \in G,$$
which shows the desired statement with $x':=pxq$.

\medskip

We now prove that the group $\mathfrak{C}$ contains an element of order $2^m$ for each $m\geq 0$ by induction on $m$. The claim trivially holds for $m=0$, so we suppose that we can choose $g\in \mathfrak{C}$ of order $2^m$ for some $m\geq 0$. Then $g':=(g, \one, \one )\in \mathfrak{C}\subset G$ by Proposition~\ref{Prop: G is branch over [G,G]}. It can be readily verified, using the induction hypothesis, that the product $g'a'=(g,\one, x')(1\, 2)\in G$
has order $2^{m+1}$, where $a'=(\one,\one, x')(1\, 2)$ is the element from the claim above. By Corollary~\ref{cor: (g,ginv,1)}, it then follows that $\big(g'a', (g'a')^{-1},\one\big)$ is an element of order $2^{m+1}$ in $\mathfrak{C}$, which completes the induction step.

It remains to construct an element of order $2^m3^n$ for some given $m,n\geq 0$. Let $g,h\in \mathfrak{C}$ be some elements of orders $2^{m}$ and $3^n$, respectively, which have been shown to exist. Then $(g,g^{-1},\one)\in \mathfrak{C}$ by Corollary~\ref{cor: (g,ginv,1)}, and $(\one,\one, h) \in  \mathfrak{C}$ by Proposition~\ref{Prop: G is branch over [G,G]}. It now follows that their product $(g,g^{-1},h)\in \mathfrak{C}\subset G$ has order $2^m3^n$, as desired. This completes the proof of the theorem.
\end{proof}

\subsection{Proof of Theorem~\ref{thm-reg-branch}} Theorem~\ref{thm-reg-branch} combines the statements of Theorems~\ref{Thm: M_r is always regular branch} and~\ref{thm-torsion}, along with Corollary~\ref{cor-weakly-branch-countable}.

\subsection{Proof of Corollary~\ref{polynomial-branch}} 
Let $f$ be a PCF cubic polynomial with coefficients in a number field $K$, such that $f$ satisfies Assumption~\ref{main-assumption}, and let $G^{geom}(f)$ be the associated profinite geometric iterated monodromy group. By the proof of Theorem~\ref{main-theorem},\ref{item-MT-1}, there exists a \ref{main-assumption}-restricted model group $G:=G^{model}(f)$ and some $w \in \AT$ such that $G^{geom}(f) = w G w^{-1}$. Set $\mathfrak{C}:=\overline{[G,G]}$ as before. It then follows from Theorem~\ref{Thm: M_r is always regular branch} that $\mathfrak C^{geom} := w \mathfrak{C} w^{-1}$ is a finite-index subgroup of $G^{geom}(f)$, and that $G^{geom}(f)$ is regular branch over it, which proves Part~\ref{item-poly-branch} of the corollary. Since torsion orders are invariant under conjugation, Part~\ref{item-poly-torsion} follows directly from Theorem~\ref{thm-torsion}.

\section{Group filtrations for cubic polynomials with three disjoint critical orbits}\label{section-filtration}

In this section, we prove Theorem~\ref{thm-filtration}.

Let $K$ be a number field, and let $f\in K(z)$ be a cubic PCF polynomial with two distinct finite critical points $c_1,c_2$. Recall that, by the Riemann-Hurwitz formula, both of these points have ramification index $2$. Furthermore, 
their obits are either periodic or strictly pre-periodic. 
Assume now that the orbits of $c_1$ and $c_2$ are disjoint. Then, in order to make sure that the ramification portrait of $f$ satisfies Assumption~\ref{main-assumption}, we only need to assume one more condition: for each $c_i$, $i=1,2$, we have that $c_i$ is a periodic point of $f$, or that both $c_i$ and $f(c_i)$ are not periodic; see Figure~\ref{fig:disjoint orbits}. In other words, if $f(c_i)$ is periodic, then $c_i$ is periodic as well. This ensures that no vertex other than $\infty$ in the ramification portrait has three incoming edges. 

Given that $f$ has the form above, we can immediately identify the ramification portrait of~$f$ with two sets of pairs of numbers $(s_i, m_i)$, $i=1,2$, where $s_i\geq 0$ denotes the length of the pre-periodic part of the orbit of $f(c_i)$, and $m_i\geq 1$ denotes the length of the respective periodic cycle; see Figure~\ref{fig:disjoint orbits} for reference.

By the proof of Theorem~\ref{main-theorem}, $G^{geom}(f)$ is conjugate to a \ref{main-assumption}-restricted model group $G$, with the model generators given by Proposition \ref{prop: model-group-from-map}. Here, one of the following two alternatives is realized for the postcritical orbit of each critical point $c_i$, $i = 1,2$:
\begin{itemize}
    \item  If the orbit of $f(c_i)$ is strictly pre-periodic, 
then the $s_i+m_i$ model generators corresponding to this orbit are given by the following system of recursive formulas:
\begin{align}\label{Strictly pre-periodic}
\begin{tabular}{@{}lll}
    $a_1$&=&$(1\,2)$ \;or \; $(2\,3)$,  \\
    $a_{s_i+1}$&=&$(a_{s_i},a_{s_i+m_i},\one)$, \\
    $a_j$&=&$(a_{j-1},\one,\one)$, \quad\quad $j\in \{2,\dots, s_i+m_i\}\setminus\{s_i+1\}$.\\
\end{tabular}
\end{align}
\item If the orbit of $f(c_i)$ is periodic, 
then the corresponding $s_i+m_i=m_i$ model generators are given by:
\begin{align}\label{Strictly periodic}
\begin{tabular}{@{}lll}
    $a_1$&=&$(a_{m_i},\one,\one)(1\,2)$\;  or \; $(\one,\one, a_{m_i})(2\,3)$,\\
    $a_j$&=&$(a_{j-1},\one,\one)$,\quad\quad  $j\in \{2,\dots,m_i\}$.\\
\end{tabular}
\end{align}
\end{itemize}
Moreover, for the postcritical orbits of $c_1$ and $c_2$ we choose distinct  options (i.e., $(1\,2)$ or $(2\, 3)$) for the first-level action of the corresponding model generator~$a_1$; compare Remark~\ref{rmk-23}.

\begin{figure}   
\begin{center}
\fbox{
    \begin{tikzpicture}[main/.style = {draw, circle}]
    
     \node (c1) at (-6,0) {$\bullet$};
    \node at (-6,.3) {$c_1$};

    \node (1) at (-5,0) {$\bullet$};
    \node at (-5,.3) {$p_1$};
    
    \node (2) at (-4,0) {$\bullet$};
    \node at (-4,.3) {$p_2$};
    
    \node (3) at (-3,0) {$\cdots$};
    
    \node (4) at (-2,0) {$\bullet$};
    \node at (-2,.3) {$p_{s_1}$};

\node (5) at (-1,0) {${\bullet}$};
\node at (-.4,0) {$p_{s_1+1}$};

\node (6) at (0,1) {${\bullet}$};
\node at (0,.6) {$p_{s_1+2}$};

\node (7) at (1,0) {\phantom{${\vdots}$}};
\node at (.9,0) {$\vdots$}; 

\node (8) at (0,-1) {${\bullet}$};
\node at (0,-.6) {$p_{s_1+m_1}$};

\draw[->] (5) .. controls (-.707,.707) .. (6);
\draw[->] (6) .. controls (.707,.707) .. (7);
\draw[->] (7) .. controls (.707,-.707) .. (8);
\draw[->] (8) .. controls (-.707,-.707) .. (5);

\draw[->] (1)--(2);
\draw[->] (2)--(3);
\draw[->] (3)--(4);
\draw[->] (4)--(5);

\draw[->] (c1) .. controls (-5.5,.2) .. (1);
\draw[->] (c1) .. controls (-5.5,-.2) .. (1);

\begin{scope}[xshift=2cm]
\node (5a) at (3,0) {${\bullet}$};
\node at (2.3,0) {$c_2=q_{1}$};

\node (6a) at (4,1) {${\bullet}$};
\node at (4,1.4) {$q_{2}$};

\node (7a) at (5,0) {\phantom{${\vdots}$}};
\node at (4.9,0) {$\vdots$};

\node (8a) at (4,-1) {${\bullet}$};
\node at (4,-.6) {$q_{m_2}$};

\draw[->] (5a) .. controls (3.293,.707) .. (6a);
\draw[->] (5a) .. controls (3.707,.293) .. (6a);
\draw[->] (6a) .. controls (4.707,.707) .. (7a);
\draw[->] (7a) .. controls (4.707,-.707) .. (8a);
\draw[->] (8a) .. controls (3.293,-.707) .. (5a);
\end{scope}
\end{tikzpicture}

}
\end{center}
\caption{Critical orbits for two typical cases: left, both $c_1$ and $f(c_1)$ are not periodic; right, $c_2$ is periodic.}
\label{fig:disjoint orbits}
\end{figure}
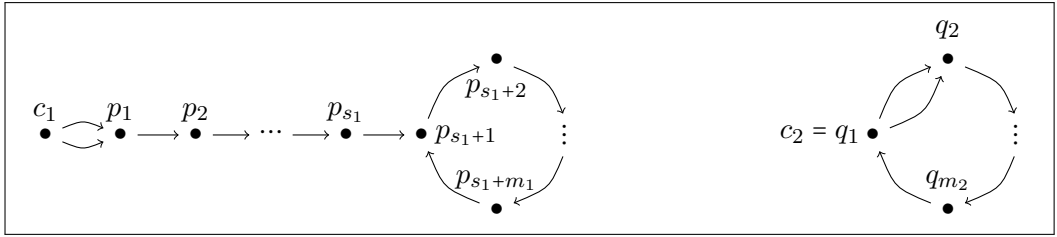

Suppose now that $f,f'\in K(z)$ are two cubic PCF polynomials, each satisfying Assumption~\ref{main-assumption} and having disjoint critical orbits. We will show below that, under certain divisibility conditions on the numbers $s_i,m_i, s'_i, m'_i$, $i=1,2$, characterising the respective ramification portraits, one can conjugate one of the given groups into a subgroup of the other one; see Theorem~\ref{thm-filtration}. In other words, up to conjugation, the profinite geometric iterated monodromy groups associated to PCF polynomials satisfying Assumption~\ref{main-assumption} and having disjoint critical orbits form \emph{filtrations} in $W=\Aut(T)$.

We begin with two technical propositions that will be used to establish the desired inclusion.

\begin{prop}\label{prop: moving generators}
    Let $m\geq 1$ and $s\geq 0$ be integers, and let $G$ be a {\normalfont\ref{main-assumption}}-restricted model group. Consider two sets of $s+m$ elements $a_1,\dots,a_{s+m}\in W$ and $b_1,\dots,b_{s+m}\in W$ given in of one of the following two recursive forms:
    \begin{enumerate}[label=(\roman*),font=\normalfont]
        \item\label{item: moving-1} $a_1=(a_{s+m},\one,\one)(1\,2)$ and $a_j=(a_{j-1},\one,\one)$ for $j\neq 1$; and \newline
        ${\,}b_1=(\one,\one,b_{s+m})(2\,3)$ and $b_j=(b_{j-1},\one,\one)$ for $j\neq 1$;   
    \end{enumerate}
    or
    \begin{enumerate}[label=(\roman*),font=\normalfont] 
    \setcounter{enumi}{1}
    \item\label{item: moving-2} $a_1=(1\,2)$, $a_{s+1}=(a_s,a_{s+m},\one)$, and $a_j=(a_{j-1},\one,\one)$ for $j\neq 1,s+1$; and\newline    
    ${\,}b_1=(2\,3)$, $b_{s+1}=(b_s,b_{s+m},\one)$, and $b_j=(b_{j-1},\one,\one)$ for $j\neq 1,s+1$.
    \end{enumerate}

    Then, for each $j=1,\dots, s+m$,  the group $G$ contains $a_j$ if and only if it contains $b_j$. Hence, $G$ contains the elements $a_1,\dots,a_{s+m}$ if and only if it also contains the elements $b_1,\dots,b_{s+m}$.
\end{prop}

\begin{proof}
    By Proposition \ref{Prop: G is branch over [G,G]}, we know that $G$ contains $[\Alt_3]^\infty$. We will show that for each $j=1,\dots,s+m$, there exists some $g_j\in [\mathcal A_3]^\infty$ such that $g_jb_j=a_j$, where we assume either of the two cases~\ref{item: moving-1} or \ref{item: moving-2}. First, we prove by induction that for each $n\geq 0$ we have  $g_{1,n},\dots, g_{s+m,n}\in [\Alt_3]^n$ such that $g_{j,n}\cdot b_j|_{T_n}=a_j|_{T_n}$ for each $j$. As this is clear for $n=0$, assume that it is true for some $n\geq 0$, giving us elements $g_{j,n}\in [\Alt_3]^n$ for $j=1,\dots,s+m$.

    Now set $g_{1,n+1}:=(g_{s+m,n},\one,\one)(1\,3\,2)\in [\Alt_3]^{n+1}$ if we are in Case~\ref{item: moving-1}, and $g_{1,n+1}:=(1\,3\,2)$, $g_{s+1,n+1}:=(g_{s,n},g_{s+m,n},\one)\in [\Alt_3]^{n+1}$ if we are Case~\ref{item: moving-2}. For all the remaining $j$'s, we set $g_{j,n+1}:=(g_{j-1,n},\one,\one)$ in both cases. It can then be readily verified that $$g_{j,n+1}\cdot b_j|_{T_{n+1}}=a_j|_{T_{n+1}}$$ for all $j=1,\dots, s+m$ in both cases, completing the induction step.

     Since $[\Alt_3]^n\subset [\Alt_3]^\infty$ for all $n\geq 0$ and the group $[\Alt_3]^\infty$ is compact, for each $j=1,\dots, s+m$, we can find a convergent subsequence of the constructed sequence $\{g_{j,n}\}$. Setting $g_j\in [\mathcal A_3]^\infty\subset G$ to be the respective limit, we conclude that $g_j\cdot b_j=a_j$. It then  follows immediately that $a_j\in G$ if and only if $b_j\in G$, finishing the proof of the proposition.
\end{proof}

\begin{prop} \label{Prop: main-filtration}
    Let $r,m\geq 1$ and $s\geq 0$ be integers such that $r$ divides both $s$ and $m$, and let $G$ be a {\normalfont\ref{main-assumption}}-restricted model group. Suppose $G$ contains $s+m$ elements $a_1,\dots, a_{s+m}$ of one of the following two recursive forms:
    \begin{enumerate}[label=(\roman*),font=\normalfont]
        \item\label{item: filtration-1} $a_1=(a_{s+m},\one,\one)(1\,2)$ and $a_j=(a_{j-1},\one,\one)$ for $j\neq 1$; or
        \item\label{item: filtration-2} $a_1=(1\,2)$, $a_{s+1}=(a_s,a_{s+m},\one)$, and $a_j=(a_{j-1},\one,\one)$ for $j\neq 1,s+1$.
    \end{enumerate} 
    Then, $G$ also contains the $r$ elements $b_1=(b_r,\one,\one)(1\,2),b_k=(b_{k-1},\one,\one)$, for $k\in \{2,\dots,r\}$.
\end{prop}

\begin{proof}
    In either case \ref{item: filtration-1} or \ref{item: filtration-2}, we define $c_k:=a_ka_{k+r}a_{k+2r}\cdots a_{k+(s+m-r)}\in G$ for each $k=1,\dots,r$. Then, again in both cases, we have $$c_1=(a_{s+m},a_ra_{2r}\cdots a_{s+m-r},\one)(1\,2)=(a_{s+m},c_ra_{s+m}^{-1},\one)(1\,2),$$ and for $k\neq 1$ we have $$c_k=(a_{k-1}a_{k-1+r}\cdots a_{k-1+(s+m-r)},\one,\one)=(c_{k-1},\one,\one).$$

    Now assume that for some $n\geq 0$ there exist $g_1,\dots,g_r\in G$ such that $(g_kc_kg_k^{-1})|_{T_n}=b_k|_{T_n}$ for all $k=1,\dots,r$. (This clearly holds when $n=0$.) Note that, by Lemma~\ref{lem: model-self-replicativity}, $G$ contains elements of the form $(g,\one,*)$ and $(g,g,*)$ for all $g\in G$. Hence, for $n+1$, we can choose $h_1:=(g_r,g_r,*)(c_ra_{s+m}^{-1},\one,*)\in G$ and $h_k:=(g_{k-1},\one,*)\in G$ for each $k\neq 1$. Then    $$(h_1c_1h_1^{-1})|_{T_{n+1}}=\big((g_rc_rg_r^{-1})|_{T_n},\one,\one\big)(1\,2)=(b_r|_{T_n},\one,\one)(1\,2)=b_1|_{T_{n+1}}$$ and $$(h_kc_kh_k^{-1})|_{T_{n+1}}=\big((g_{k-1}c_{k-1}g_{k-1}^{-1})|_{T_n},\one,\one\big)=(b_{k-1}|_{T_n},\one,\one)=b_k|_{T_{n+1}},$$ for $k\neq 1$. By induction, we conclude that $b_k|_{T_n}\in G_n$ for all $n\geq 0$ and $k=1,\dots, r$, and therefore $b_k\in G$ for all $k$.
\end{proof}

Let us illustrate the two propositions above by comparing several specific \ref{main-assumption}-restricted model groups. 

\begin{ex}\label{ex: filtrations-model-groups}
    Consider the elements $a,a_1,a_2,b,b_1,b_2\in W$ defined by the following recursive formulas: 
    \[
    \begin{tabular}{@{}lll}
    $a=(a,\one,\one)(1\,2)$, &$a_1=(a_2,\one,\one)(1\,2)$,  &$a_2=(a_1,\one,\one)$,\quad and\\
    $b=(\one,\one,b)(2\,3)$,  &$b_1=(\one,\one,b_2)(2\,3)$, & $b_2=(\one,\one,b_1)$. 
    \end{tabular}
    \]

    Then we have the following relations between the induced \ref{main-assumption}-restricted model groups:

    \[\llangle a,b\rrangle\subsetneq\llangle a,b_1,b_2\rrangle=\llangle a_1,a_2,b\rrangle=\llangle a_1,a_2,b_1,b_2\rrangle.\]

    All the equalities above and the non-strict inclusion $\llangle a,b\rrangle\subseteq\llangle a,b_1,b_2\rrangle$ follow directly from Propositions~\ref{prop: moving generators} and~\ref{Prop: main-filtration}. The latter inclusion can be seen to be strict using the following trick. For each $n\geq 0$, let us define a map $\Upsilon_n\colon W\to \{\pm1\}$ by $\Upsilon_n(g):=\prod_{v\in \sX^n}\sgn(g|_v)$, where $\sgn\colon W\to \{\pm1\}$ is the map sending each $h\in W$ to the sign of the permutation $h|_{T_1}\in \Sym_3$. As each $\Upsilon_n$ is a group homomorphism, the map $\Upsilon\colon W \to \{\pm1\}\times \{\pm 1\}$, $g\mapsto \big(\Upsilon_0(g),\Upsilon_1(g)\big)$, is also a group homomorphism. Now note that $\Upsilon(a)=(-1,-1)=\Upsilon(b)$, but $\Upsilon(a_1)=(-1,1)$ and $\Upsilon(a_2)=(1,-1)$. It follows that $a,b$ cannot generate $b_1,b_2$.
\end{ex}

    \begin{proof}[Proof of Theorem~\ref{thm-filtration}]
    Let $f,f'$ be the PCF cubic polynomials as in the statement, and let $(s_i,m_i), (s'_i, m'_i)$, $i=1,2$, be the corresponding sets of pairs of numbers characterizing the ramification portraits of $f,f'$, respectively. Let us further assume that these numbers satisfy the hypotheses from the theorem. By the discussion above, the profinite geometric iterated monodromy groups $G^{geom}(f)$ and $G^{geom}(f')$ are conjugate in $W$ to some \ref{main-assumption}-restricted model groups $G$ and $G'$, respectively, whose model generators follow the recursive formulas from  \eqref{Strictly pre-periodic} or \eqref{Strictly periodic}, for the appropriate values of the parameters $(s_i,m_i), (s'_i, m'_i)$, $i=1,2$. The hypotheses on these parameters imply that the group $G$ contains all the model generators of the group $G'$ by Propositions~\ref{prop: moving generators} and~\ref{Prop: main-filtration}. Hence, we have the inclusion $G'\subseteq G$ of the model groups, and thus $wG^{geom}(f')w^{-1}\subseteq G^{geom}(f)$ for some element $w \in W$.
    \end{proof}

\begin{ex}
    In the setting of Theorem~\ref{thm-filtration}, suppose that $s_1=s_2=s_1'=s_2'=0$. If $m_1=2$ and $m_2=m_1'=m_2'=1$, then we have seen that $wG^{geom}(f')w^{-1} \subseteq G^{geom}(f)$ for some $w\in W$. Moreover, this inclusion is strict for all such $w$ by Example~\ref{ex: filtrations-model-groups}. Now, if we take instead, say, $m_1=m_2=m_1'=2$ and $m_2'=1$, then we have the equality $wG^{geom}(f')w^{-1} = G^{geom}(f)$ for some $w\in W$ by the same example. 
\end{ex}

\end{document}